\documentclass[12pt]{amsart}
\usepackage[english]{babel}

\usepackage[letterpaper,top=2cm,bottom=2cm,left=3cm,right=3cm,marginparwidth=1.75cm]{geometry}

\usepackage{amsmath}
\usepackage{amsthm}
\usepackage{amssymb}
\usepackage{graphicx}
\usepackage{mathtools}
\usepackage{comment}
\usepackage{tikz-cd} 

\usepackage{physics2}
\usephysicsmodule{ab}
\usephysicsmodule{ab.braket}    

\theoremstyle{plain}
\newtheorem{thm}{Theorem}[section]
\newtheorem{lmm}[thm]{Lemma}
\newtheorem{prp}[thm]{Proposition}
\newtheorem{cor}[thm]{Corollary}

\theoremstyle{definition}
\newtheorem{dfn}[thm]{Definition}
\newtheorem{rmk}[thm]{Remark}
\newtheorem{ex}[thm]{Example}

\usepackage{color}
\definecolor{e-mail}{rgb}{0,.40,.80}
\definecolor{reference}{rgb}{.20,.60,.22}
\definecolor{mrnumber}{rgb}{.80,.40,0}
\definecolor{citation}{rgb}{0,.40,.80}
\usepackage[breaklinks=true,colorlinks=true,linkcolor=reference, citecolor=citation, urlcolor=e-mail]{hyperref}

\newcommand{\p}{\mathfrak{p}}
\newcommand{\E}{\mathrm{ell}}
\newcommand{\et}{\mathrm{\acute{e}t}}
\newcommand{\gal}{\mathrm{Gal}}
\newcommand{\cyc}{\mathrm{cyc}}
\newcommand{\fracp}{\lbrack 1/p \rbrack}

\title[]{On factorizations of certain Kummer characters associated to once-punctured elliptic curves with complex multiplication}
\date{}

\email{ishii.shun@keio.jp}

\author[]{Shun Ishii$^{\dagger}$}
\address[$\dagger$]{{\footnotesize Department of Mathematics, Keio University, 3-14-1 Hiyoshi, Kouhoku-ku, Yokohama 223-8522, Japan.}}

\author[]{Yuki Goto$^{\dagger\diamond}$}
\address[$\diamond$]{{\footnotesize Mathematical Science Team, RIKEN Center for Advanced Intelligence Project (AIP),1-4-1 Nihonbashi, Chuo-ku, Tokyo 103-0027, Japan.}}

\begin{document}

\begin{abstract}
In this paper, we study certain Kummer characters, which we call the elliptic Soul\'e characters, arising from Galois actions on the pro-$p$ fundamental groups of once-punctured elliptic curves with complex multiplication. In particular, we prove that elliptic Soul\'e characters having values in Tate twists can be written in terms of the Soul\'e characters and generalized Bernoulli numbers. We apply this result to give a criterion for surjectivity of the elliptic Soul\'e characters and an analogue of the Coleman-Ihara formula.
\end{abstract}

\maketitle

\begin{center}
\tableofcontents
\end{center}

\section{Introduction}\label{1}

Let $K$ be an imaginary quadratic field of class number one, $E$ an elliptic curve over $K$ with complex multiplication by the ring of integers $O_{K}$ of $K$ and $p \geq 5$ be a prime which splits in $K$ as $(p)=\p \bar{\p}$. The purpose of this paper is to investigate certain Kummer characters (cf. Definition \ref{dfn:ellSoule}) associated to $p$-power roots of special values of the fundamental theta function (cf. Definition \ref{dfn:theta}) of $E$ and prove a certain factorization formula for such characters (cf. Theorem \ref{thm:main}). 

Geometrically, such characters are known to arise from (outer) Galois actions on the pro-$p$ geometric fundamental groups of the once-punctured elliptic curve $X=E \setminus O$
\[
    G_{K} \coloneqq \gal(\overline{K}/K) \to \mathrm{Out}(\pi_{1}(X \times_{K} \overline{K})^{(p)}) \coloneqq \mathrm{Aut}(\pi_{1}(X \times_{K} \overline{K})^{(p)})/\mathrm{Inn}(\pi_{1}(X \times_{K} \overline{K})^{(p)})
\] by Nakamura \cite[Theorem A]{Na95} (see also Remark \ref{rmk:Nakamura}). Here, $\overline{K}$ is an algebraic closure of $K$ and $\pi_{1}(X_{\overline{K}})^{(p)}$ is the maximal pro-$p$ quotient of the \'etale fundamental group of $X \times_{K} \overline{K}$ \footnote{More precisely, such characters arise from Galois actions on the maximal \emph{metabelian} quotient of the pro-$p$ fundamental group $\pi_{1}(X \times_{K} \overline{K})^{(p)}$. }. 

In the following of this paper, we also fix an embedding from $\overline{K}$ into $\mathbb{C}$ and a Weierstrass form $y^{2}=4x^{3}-g_{2}x-g_{3}$ of $E$ over $K$. Let $L=\Omega \cdot O_{K}$ be the lattice which gives a uniformization 
\[
    \mathbb{C}/L \xrightarrow{\sim} E(\mathbb{C}) \; : \; z \mapsto [\wp_{L}(z) : \wp_{L}'(z) : 1].
\] of the fixed Weierstrass form, where $\wp_{L}(z)$ is the usual Weierstrass $\wp$-function. For each $n\geq 1$, the element $\Omega/p^{n} \in \mathbb{C}$ gives an $O_{K}$-basis $\omega_{n}$ of $E[p^{n}](\overline{K})$ under this uniformization. Moreover, we have an isomorphism
\[
E[p^{n}](\overline{K}) \xrightarrow{\sim} E[\p^{n}](\overline{K}) \oplus E[\bar{\p}^{n}](\overline{K})
\] induced by $
O_{K}/(p^{n}) \xrightarrow{\sim} O_{K}/(\p^{n}) \times O_{K}/(\bar{\p}^{n})
$.  We write the image of $\omega_{n}$ under this isomorphism by $(\omega_{\p,n}, \omega_{\bar{\p},n})$. We have two characters
\[
\varepsilon_{\p} \colon G_{K} \to O_{K,\p}^{\times} \xrightarrow{\sim} \mathbb{Z}_{p}^{\times} 
\quad \text{and} \quad
\varepsilon_{\bar{\p}} \colon G_{K} \to O_{K,\bar{\p}}^{\times} \xrightarrow{\sim}  \mathbb{Z}_{p}^{\times} 
\] associated to the $\p$-adic Tate module $T_{\p}E$ and  the $\bar{\p}$-adic Tate module $T_{\bar{\p}}E$, respectively. For a pair of integers $\boldsymbol{m}=(m_{1}, m_{2})$, we write
\[
\varepsilon^{\boldsymbol{m}} \coloneqq \varepsilon_{\p}^{m_{1}}\varepsilon_{\bar{\p}}^{m_{2}} \colon G_{K} \to \mathbb{Z}_{p}^{\times}.
\] With this notation, $\varepsilon^{(1,1)}$ coincides with the $p$-adic cyclotomic character $\varepsilon_{\cyc}$ on $K$.

We denote the mod-$p^{n}$ ray class field of $K$ by $K(p^{n})$ and let $
K(p^{\infty})\coloneqq \cup_{n>0} K(p^{n})$, which is the maximal abelian extension of $K$ unramified outside $p$. Note that the field $K(p^{\infty})$ corresponds to the kernel of the homomorphism
\[
(\varepsilon_{\p}, \varepsilon_{\bar{\p}}) \bmod O_{K}^{\times} \colon G_{K} \to \mathrm{Aut}_{O_{K} \otimes \mathbb{Z}_{p}}\left( T_{p}E \right) / O_{K}^{\times} \cong \left( O_{K,\p}^{\times} \times O_{K, \bar{\p}}^{\times} \right)/O_{K}^{\times}.
\] In fact, the fixed field of the kernel is obtained by adjoining all coordinates of $E[p^{\infty}](\overline{K})/O_{K}^{\times} \subset E/O_{K}^{\times} \cong \mathbb{P}^{1}_{K}$ to $K$ where the isomorphism is given by a Weber function \cite[Example 5.5.1]{Si}. The desired result follows from \cite[Theorem 5.6]{Si}.

In particular, the character $\varepsilon^{\boldsymbol{m}}$ factors through $\gal(K(p^{\infty})/K)$ if $m_{1} \equiv m_{2} \bmod w_{K}$, where $w_{K} \coloneqq |O_{K}^{\times}|$ is the order of the unit group of $O_{K}$. For a pair of positive integers $\boldsymbol{m}=(m_{1}, m_{2})$ and $n \geq 1$, we define \[
\epsilon_{\boldsymbol{m}, n} \coloneqq \prod_{\substack{0 \leq a,b<p^{n} \\ p \nmid  \gcd(a,b)}}  \theta(a\omega_{\p, n}+b\omega_{\bar{\p}, n}, L)^{a^{m_{1}-1}b^{m_{2}-1}} 
\in K(p^{\infty})^{\times}/K(p^{\infty})^{\times p^{n}},
\] where $\theta(z,L)$ is the fundamental theta function associated to $L$, cf. Definition \ref{dfn:theta}. 

Note that, if we replace the generator $\Omega$ of $L$ with a $O_{K}^{\times}$-multiple of it, then both $\omega_{\p,n}$ and $\omega_{\bar{\p},n}$ are multiplied by such the unit. However, the element $\epsilon_{\boldsymbol{m}, n}$ does not depend on the choice of $\Omega$, since the fundamental theta function satisfies $\theta(cz,  cL)=\theta(z,L)$ for every $c \in \mathbb{C}^{\times}$.

\begin{rmk}
Since the fundamental theta function is not periodic with respect to $L$, the value $\theta(a\omega_{\p, n,1}+b\omega_{\bar{\p}, n}, L)$ is not well-defined as it is. However, in Lemma \ref{lmm:theta_val} (1), we will show that the value $\theta(\omega, L)$ is contained in $K(p^{\infty})^{\times}$ for every $\omega \in \frac{1}{p^{n}}L \setminus L$ which lifts $a\omega_{\p, n}+b\omega_{\bar{\p}, n} \in \frac{1}{p^{n}}L/L$, and its image in the $p$-adic completion $\widehat{K(p^{\infty})^{\times}}$ does not depend on the choice of $\omega$. Hence we define the value 
\[
\theta(a\omega_{\p, n}+b\omega_{\bar{\p}, n}, L) \coloneqq \theta(\omega, L) \in K(p^{\infty})^{\times}/K(p^{\infty})^{\times p^{n}}
\] by taking an arbitrary lift $\omega$, cf. Lemma \ref{lmm:theta_val} (2).
\end{rmk}

\begin{lmm}[cf. Section \ref{2.1} for the proof]\label{lmm:epsilon} For each $n \geq 1$, the following assertions hold.
\begin{enumerate}
    \item The element $\epsilon_{\boldsymbol{m}, n}$ is invariant under the action of $\gal(K(p^{\infty})/K(p^{n}))$. Moreover, we have \[
        \sigma(\epsilon_{\boldsymbol{m}, n})=(\epsilon_{\boldsymbol{m}, n})^{\varepsilon_{\cyc}\varepsilon^{-\boldsymbol{m}}(\sigma)}
    \] for every $\sigma \in \gal(K(p^{\infty})/K)$.
    \item The relation
    \[
    \epsilon_{\boldsymbol{m}, n+1}=\epsilon_{\boldsymbol{m}, n}
    \] holds in $K(p^{\infty})^{\times}/K(p^{\infty})^{\times p^{n}}$. 
\end{enumerate}
\end{lmm}

Lemma \ref{lmm:epsilon} allows us to define elliptic Soul\'e characters. In the following, let $\zeta_{p^{n}} \coloneqq \exp(\frac{2 \pi i}{p^{n}}) \in \mathbb{C}$ be the primitive $p^{n}$-th root of unity.

\begin{dfn}\label{dfn:ellSoule}
For a pair of positive integers $\boldsymbol{m}=(m_{1}, m_{2})$, we define the $\boldsymbol{m}$-th elliptic Soul\'e character 
\[
\chi_{\boldsymbol{m}}^{\E} \colon G_{K(p^{\infty})} \to \mathbb{Z}_{p}(\boldsymbol{m}) \coloneqq \mathbb{Z}_{p}(\varepsilon^{\boldsymbol{m}})
\] to be a unique character satisfying the following relation for every $n \geq 1$ and $\sigma \in  G_{K(p^{\infty})}$: 
\[
\zeta_{p^{n}}^{\chi_{\boldsymbol{m}}^{\E}(\sigma)}=\frac{\sigma \left( \epsilon_{\boldsymbol{m}, n}^{\frac{1}{p^{n}}}  \right)}{ \epsilon_{\boldsymbol{m}, n}^{\frac{1}{p^{n}}}}.
\]
\end{dfn}

Since the element $\epsilon_{\boldsymbol{m}, n}$ does not depend on the choice of $\Omega$, neither does $\chi_{\boldsymbol{m}}^{\E}$.

\begin{rmk}[Relation to Nakamura's work {\cite{Na95}}]\label{rmk:Nakamura}
  We remark that essentially the same characters are already introduced and studied in the pioneering work of Nakamura \cite{Na95}. If we fix a $\mathbb{Z}$-basis ${\omega_{1},\omega_{2}}$ of $L$ such that  the ratio $\omega_{1}/\omega_{2}$ is contained in the upper half-plane, then the character 
    \[
        \kappa_{i,j} \colon G_{K(E[p^{\infty}])} \to \mathbb{Z}_{p}
    \] in \cite[(3.11.5)]{Na95} is defined in the same manner as $\chi^{\E}_{i+1,j+1}$ by using 
    \[
        \prod_{\substack{0 \leq a,b<p^{n} \\ p \nmid  \gcd(a,b)}}  \theta(a\omega_{1}+b\omega_{2}, p^{n}L)^{a^{i}b^{j}}
    \] instead of $\epsilon_{\boldsymbol{m},n}$. Nakamura proved \cite[Theorem A]{Na95} that these characters describe the $G_{K(E[p^{\infty}])}$-action on the commutator subgroup of the maximal metabelian quotient of $\pi_{1}(X \times_{K} \overline{K})^{(p)}$. 
    
    The only difference between $\kappa_{i,j}$ and $\chi^{\E}_{i+1,j+1}$ is the choice of a basis of the $p$-adic Tate module $T_{p}E$. Indeed, by a direct computation, one can observe that
    \[
        \sum_{i+j=m}\kappa_{i,j} \cdot \omega_{1}^{\otimes i} \otimes \omega_{2}^{\otimes j} \otimes (\zeta_{p^{n}})_{n} \in H^{1}(G_{K(E[p^{\infty}])}, \mathrm{Sym}^{m}T_{p}E(1))^{\gal(K(E[p^{\infty}])/K)}
    \] coincides with the restriction of $\sum_{i+j=m} \chi^{\E}_{i+1,j+1} \cdot \omega_{\p}^{\otimes i} \otimes \omega_{\bar{\p}}^{\otimes j} \otimes (\zeta_{p^{n}})_{n}$ to $K(E[p^{\infty}])$, noting that our characters are defined on $G_{K(p^{\infty})}^{\mathrm{ab}}$. Therefore, the characters we consider are the same as those in \cite{Na95}, once we regard both of them as $\mathrm{Sym}^{m}T_{p}E(1)$-valued homomorphisms. We remark that Nakamura treats \emph{every} once-punctured elliptic curves over number fields, not only those having complex multiplication. Interested readers are encouraged to read the original paper, which contains applications to anabelian geometry for genus one hyperbolic curves.
\end{rmk}

The elliptic Soul\'e characters are analogues of the following so-called Soul\'e characters: For each integers $m, n \geq 1$, we define 
\[
\epsilon_{m,n}^{\cyc} \coloneqq \prod_{\substack{0 < a <p^{m} \\ p \nmid a }}
(1-\zeta_{p^{n}}^{a})^{a^{m-1}}  \in K(\mu_{p^{n}})^{\times}/K(\mu_{p^{n}})^{\times p^{n}}.
\] One can confirm the following analogue of Lemma \ref{lmm:epsilon}: 
\begin{itemize}
    \item $
    \sigma(\epsilon_{m, n}^{\cyc})=(\epsilon_{m, n}^{\cyc})^{\varepsilon_{\cyc}^{1-m}(\sigma)}
    $ for every $\sigma \in \gal(K(\mu_{p^{n}})/K)$.
    \item  
    $
    \epsilon_{m, n+1}^{\cyc}=\epsilon_{m, n}^{\cyc}
    $ holds in $K(\mu_{p^{n+1}})^{\times}/K(\mu_{p^{n+1}})^{\times p^{n}}$.
\end{itemize}

\begin{dfn}[Soul\'e characters]\label{dfn:Soule}
For each odd integer $m \geq 3$, we define the $m$-th Soul\'e character 
\[
\chi_{m} \colon G_{K(\mu_{p^{\infty}})} \to \mathbb{Z}_{p}(m) \coloneqq \mathbb{Z}_{p}(\varepsilon_{\cyc}^{m})
\] to be a unique character satisfying
\[
\zeta_{p^{n}}^{\chi_{m}(\sigma)}=\frac{\sigma \left( \epsilon_{m, n}^{\frac{1}{p^{n}}}  \right)}{ \epsilon_{m, n}^{\frac{1}{p^{n}}}}
\quad \text{for every $n \geq 1$ and $\sigma \in  G_{K(\mu_{p^{\infty}})}$}.
\]  
\end{dfn}

\begin{rmk}\label{rmk:Soule}\,
    \begin{enumerate}
    \item The $m$-th Soul\'e character extends to an element of  $H^{1}_{\et}(\mathbb{Z}[1/p], \mathbb{Z}_{p}(m))$ (cf. Section \ref{2.3}). 

    \item If we define $\chi_{m}$ for even $m \geq 2$ in a similar way, then $\chi_{m}$ is trivial (cf. Example \ref{ex:Soule}).

    \item The $m$-th Soul\'e character is nontrivial. Similarly, the first author proved that their elliptic counterparts are nontrivial under suitable assumptions \cite[Theorem 1.5 (1)]{Is23}. For example, $\chi^{\E}_{(m,m)}$ is nontrivial for every $m \geq 2$ \cite[Corollary 4.10]{Is23}. 
    \end{enumerate}
\end{rmk}

The main result of this paper is the following explicit formula between the elliptic Soul\'e characters and the Soul\'e characters. In the rest of this paper, we denote $\chi^{\E}_{(m,m)}$ by  $\chi^{\E}_{m}$ for simplicity. 

\begin{thm}[cf. Corollary \ref{cor:main2}]\label{thm:main}
Let $m \geq 2$ be an integer, $\chi_{K}$ the Dirichlet character associated to $K$, $B_{m, \chi_{K}}$ the $m$-th generalized Bernoulli number of $\chi_{K}$ \cite[Chapter 4, p.31]{Wa82} and $\chi_{m}^{(d_{K})} \colon G_{K(\mu_{p^{\infty}})} \to \mathbb{Z}_{p}(m)$ a certain variant of Soul\'e characters, cf. Definition \ref{dfn:genSoule}. Then the following relation holds between characters on $G_{K(p^{\infty})}$:
\[
\frac{\chi_{m}^{\E}}{1-p^{2m-2}}
=
\begin{dcases}
-6w_{K} \cdot \frac{B_{m}}{m} \cdot \frac{\chi_{m}^{(d_{K})}}{1-p^{m-1}} & \text{if $m$ is even,} \\
-3w_{K} \cdot \frac{B_{m,\chi_{K}}}{m} \cdot \frac{\chi_{m}}{1-p^{m-1}} & \text{if $m$ is odd.}
\end{dcases}
\]
\end{thm}

\begin{rmk}
Our result may be regarded as an \'etale analogue of the classical factorization $
\zeta_{K}(s)=\zeta(s)L(s, \chi_{K}) 
$ of the Dedekind zeta function of $K$ into the product of two $L$-functions. Note that there is a well-known $p$-adic analogue of this formula established by Gross \cite[Theorem]{Gr80}, stating that Katz's two-variable $p$-adic $L$-function factors into the product of two Kubota-Leopoldt's $p$-adic $L$-functions when restricted on the cyclotomic line. 
\end{rmk}

We give two corollaries of Theorem \ref{thm:main}. The first one  relates the non-vanishing of elliptic Soul\'e characters modulo $p$ to the $p$-indivisibility of the class number of $K(\mu_{p})$. This result is an analogue of the fact that the cyclotomic Soul\'e character $\chi_{m}$ is surjective for every $m=3,5,\dots,p-2$ if and only if Vandiver's conjecture holds for $p$ \cite[Corollary of Proposition 4]{IS87}.

\begin{cor}[cf. Corollary \ref{cor:main3}]
    The following two assertions are equivalent.
    \begin{enumerate}
        \item $\chi_{m}^{\E}$ is surjective for every $m = 2,3, \dots, p-2$.
        \item The class number of $K(\mu_{p})$ is not divisible by $p$.
    \end{enumerate}
\end{cor}

The second is an analogue of the so-called Coleman-Ihara formula \cite[p.105, (Col2)]{Ih86}:

\begin{cor}[cf. Corollary \ref{cor:Coleman-Ihara}] We have
\[
\frac{\chi_{m}^{\E}(\mathrm{rec}(\boldsymbol{u}))}{1-p^{2m-2}}
=
\begin{dcases}
3w_{K} \cdot  L^{\mathrm{Katz}}_{p}(m, \omega^{1-m}) \frac{d_{K}^{m}}{g(\chi_{K})} \frac{\phi^{\mathrm{CW}}_{m}(\boldsymbol{u})}{p^{m-1}-1} & \text{if $m$ is even,} \\
3w_{K} \cdot L^{\mathrm{Katz}}_{p}(m, \omega^{1-m}) \frac{\phi^{\mathrm{CW}}_{m}(\boldsymbol{u})}{p^{m-1}-1} & \text{if $m$ is odd.} 
\end{dcases}
\] Here, $\mathrm{rec} \colon U_{\infty} \to G_{\mathbb{Q}_{p}(\mu_{d_{K}p^{\infty}})}^{\mathrm{ab}}$ is the local reciprocity map, $L^{\mathrm{Katz}}_{p}$ is Katz's  $p$-adic L-function restricted to the cyclotomic line, $g(\chi_{K}) \coloneqq \sum_{a}\chi_{K}(a)\zeta_{d_{K}}^{-a}$ and $\phi^{\mathrm{CW}}_{m} \colon U_{\infty} \to \mathbb{Z}_{p}(m)$ is the $m$-th Coates-Wiles homomorphism (see the discussion before Corollary \ref{cor:Coleman-Ihara} for definitions).
\end{cor}

We explain the strategy to prove Theorem \ref{thm:main}: The critical ingredient is Kersey's theorem \cite[Theorem 4.1]{Ke80}, which compares absolute values of cyclotomic units and elliptic units, cf. Theorem \ref{thm:Kersey}. There the images of elliptic units under a certain regulator map $\rho$ (cf. Section \ref{2.2}) are described in terms of cyclotomic units by specializing the regulator map at each (complex-valued) character. We use his result to trace a somewhat ``reverse'' direction of Kersey's theorem: That is, we use Kersey's theorem to express images of elliptic units under a certain \'etale regulator map $\rho_{\et}$ (cf. Section \ref{2.3}) in terms of cyclotomic units to prove Theorem \ref{thm:main_measure}, which is a certain congruence between cocycle classes. Theorem \ref{thm:main} is then obtained by specializing Theorem \ref{thm:main_measure} at powers of the $p$-adic cyclotomic character. 

This paper is organized as follows: In Section \ref{2.1}, we recall some facts on the fundamental theta function. In Section \ref{2.2}, we introduce Kersey's theorem and perform some computations used in the following section. Then we introduce a certain regulator map $\rho_{\et}$ in Section \ref{2.3}. In Section \ref{3.1}, we prove Theorem \ref{thm:main_measure} and Theorem \ref{thm:main} along the strategy explained above. We give proofs of two corollaries in the last Section \ref{3.2}.

\subsection*{Notation.}

\begin{itemize} 
\item As was mentioned, we fix an algebraic closure $\overline{K}$ of $K$ and an embedding $\overline{K} \hookrightarrow \mathbb{C}$. For a subfield $F \subset \mathbb{C}$ and $n \geq 1$, we denote its ring of integers by $O_{F}$, the group of roots of unity in $F$ by $\mu(F)$ and the group of $n$-th roots of unity in $\mathbb{C}$ by $\mu_{n}$. The group $\mu_{n}$ is generated by $\zeta_{n} \coloneqq e^{\frac{2 \pi i}{n}}$. For a finite extension $L/F$, we denote the norm map $L \to K$ by $\mathrm{N}_{L/K}$. For a nonzero ideal $\mathfrak{a}$ of $O_{K}$, we denote the cardinality of the quotient ring $O_{K}/\mathfrak{a}$ by $\mathrm{N}\mathfrak{a}$. For $\sigma \in G_{K}$ and $\alpha \in \overline{K}$, we often write $\sigma\alpha$ as $\alpha^{\sigma}$.

\item $B_{1}(x) \coloneqq x-\frac{1}{2} \in \mathbb{Q}[x]$ denotes the first Bernoulli polynomial. For an element $a \in \mathbb{Q}/\mathbb{Z}$, $\Bab{a}  \in \mathbb{Q}$ denotes its fractional part, i.e.  a unique representative of $a$ satisfying both $0 \leq \Bab{a} <1$ and $a \equiv \Bab{a} \bmod  \mathbb{Z}$. 
\end{itemize}

\section{Preliminaries}\label{2}

\subsection{Fundamental theta function}\label{2.1}

In this section, we recall some basic results on the fundamental theta function \cite[Chapter II, 2.1 (7)]{dS87} and give a proof of Lemma \ref{lmm:epsilon}. First, let
\[
A(L) \coloneqq \frac{\mathrm{Area}(\mathbb{C}/L)}{\pi} \quad \text{and} \quad
s_{2}(L) \coloneqq \lim_{s \to +0} \sum_{0 \neq \omega \in L} \frac{1}{\omega^{2} | \omega |^{2s} }.
\] We define a $\mathbb{R}$-linear function $\eta(z, L) \colon \mathbb{C} \to \mathbb{C}$ to be 
$
\eta(z, L) \coloneqq A(L)^{-1} \bar{z}+s_{2}(L)z.
$

\begin{dfn}[cf. de Shalit {\cite[Chapter II, 2.1 (7) and 2.3]{dS87}}]\label{dfn:theta}\,
    \begin{enumerate}
    \item We define the fundamental theta function associated to $L$ as
    \[
    \theta(z, L) \coloneqq \Delta(L) \cdot e^{-6\eta(z, L)z} \cdot \sigma(z, L)^{12}
    \] where $\Delta(L)$ is the discriminant of the lattice $L$ and $\sigma(z, L)$ is the usual Weierstrass $\sigma$-function 
    \[
    \sigma(z, L)=z \prod_{0 \neq \omega \in L}(1-\frac{z}{\omega})\exp \left( \frac{z}{\omega}+\frac{1}{2}\left( \frac{z}{\omega} \right)^{2} \right).
    \]
    \item Let $\mathfrak{a}=(\alpha)$ be a nontrivial ideal of $O_{K}$. We define a rational function $\theta_{\mathfrak{a}}$ on $E$ by
    \[
    \theta_{\mathfrak{a}}(z) \coloneqq  
    \frac{\theta(z, L)^{\mathrm{N}\mathfrak{a}}}{\theta(z, \mathfrak{a}^{-1}L)}
    =
    \alpha^{-12}\Delta(L)^{\mathrm{N}\mathfrak{a}-1}\prod_{P \in E[\mathfrak{a}]\setminus O}\left( x(z)-x(P) \right)^{-6}.
    \] Here, $x \colon E \to \mathbb{P}^{1}$ is a finite morphism corresponding to the $x$-coordinate of $E$. Note that the function $\theta_{\mathfrak{a}}$ is defined over $K$.
    \end{enumerate}
\end{dfn}

\begin{rmk}
The function $e^{\frac{\eta(z, L)z}{2}} \cdot \sigma(z, L)$ is called \emph{Klein form} (associated to $L$) and denoted by $\mathfrak{k}(z, L)$ in Kubert-Lang \cite[p. 27]{KL81}. According to their notation, the fundamental theta function can be written as $
\theta(z, L) = \Delta(L) \cdot \mathfrak{k}(z, L)^{12}
$, and this also coincides with the $12$-th power of the \emph{Siegel function} in their book.
\end{rmk}

First, we observe the following:

\begin{lmm}\label{lmm:theta_val}
    Let $\omega \in E[p^{\infty}](\overline{K})$ be a nonzero $p$-primary torsion point of $E$. Then the following assertions hold.  
    \begin{enumerate}
        \item Let $\widetilde{\omega} \in \mathbb{C}$ be an arbitrary lift of $\omega$. Then $\theta(\widetilde{\omega}, L)$ is contained in $K(p^{\infty})^{\times}$.
        \item Let  $\widetilde{\omega}$ and $\widetilde{\omega}'$ be arbitrary lifts of $\omega$. Then the ratio
        \[
        \frac{\theta(\widetilde{\omega}, L)}{\theta(\widetilde{\omega}', L)}
        \] is contained in $\mu_{p^{\infty}}$. In particular, the value of the fundamental theta function 
        \[
        \theta(\widetilde{\omega}, L) \in \widehat{K(p^{\infty})^{\times}} \coloneqq  \varprojlim_{n} K(p^{\infty})^{\times}/K(p^{\infty})^{\times p^{n}}
        \] in the $p$-adic completion is independent of the choice of a lift. In the following, we denote this element by $\theta(\omega, L)$. With this notation, we have 
        $
         \theta(\omega_{n}, L)=\theta(\Omega/p^{n}, L).
        $
        \item Let $\mathfrak{a}$ be a nontrivial ideal of $O_{K}$ which is prime to $p$ and $(\frac{K(p^{\infty})/K}{\mathfrak{a}}) \in \gal(K(p^{\infty})/K)$ the corresponding Artin symbol. Then we have
        \[
            \left( \frac{K(p^{\infty})/K}{\mathfrak{a}} \right)\left( \theta(\omega, L)
            \right)
            =\theta(\omega, \mathfrak{a}^{-1}L).
        \] In particular, if $\omega$ is annihilated by $p^{n}$ for some $n \geq 1$, then $\theta(\omega, L)$ is contained in the $\gal(K(p^{\infty})/K(p^{n}))$-invariant subspace of $\widehat{K(p^{\infty})^{\times}}$.
        \item  If $p \cdot \omega \neq 0$, then we have
        \[
            \prod_{P \in E[p](\overline{K})} \theta(\omega+P, L)=\theta(p \cdot \omega, L).
        \]
    \end{enumerate}
\end{lmm}

\begin{proof}
    (1) The assertion that $\theta(\widetilde{\omega}, L) \in K(E[p^{\infty}])^{\times}$ is proved in \cite[(2.9)]{Na95} for general elliptic curves over number fields. Moreover, by applying \cite[Chapter 11, Theorem 1.1]{KL81} to the Fricke family associated to Siegel functions \cite[Chapter 11, p.235]{KL81}, we also obtain that $
    \theta(\widetilde{\omega}, L)^{p^{n}} \in K(p^{\infty})^{\times}
    $ if $p^{n}\widetilde{\omega} \in L$. Now, if $\theta(\widetilde{\omega}, L)$ is not contained in $K(p^{\infty})^{\times}$, then $p$ would divide the degree $[K(E[p^{\infty}]) : K(p^{\infty})]$ by Kummer theory. However, since the Galois group $\gal(K(E[p^{\infty}])/K(p^{\infty}))$ injects into $O_{K}^{\times} \subset \mathrm{Aut}(T_{p}E)$, its order is prime to $p \geq 5$. This proves the first assertion.
    
    (2) The assertion follows from \cite[p. 28, {\bf K2}]{KL81}. (3) Let $n \geq 1$ be an integer such that $p^{n}$ annihilates $\omega$. By \cite[p. 236]{KL81}, we have $\left( \frac{K(p^{\infty})/K}{\mathfrak{a}} \right) \theta(
    \widetilde{\omega}, L)^{p^{n}}
    =
    \theta(\widetilde{\omega}, \mathfrak{a}^{-1}L)^{p^{n}}
    $ in $K(p^{\infty})^{\times}$, which implies the desired equality modulo $\mu_{p^{\infty}}$. Since the image of $\mu_{p^{\infty}}$ is trivial in $\widehat{K(p^{\infty})^{\times}}$, the assertion follows. (4) The distribution relation \cite[p. 43, Theorem 4.1]{KL81} implies the desired equality modulo $\mu_{p^{\infty}}$. The assertion follows from the same argument as the proof of (3).
\end{proof}

Now we give a proof of  Lemma \ref{lmm:epsilon}. Recall that we have defined
\[
\epsilon_{\boldsymbol{m}, n} \coloneqq \prod_{\substack{0 \leq a,b<p^{n} \\ p \nmid  \gcd(a,b)}}  \theta(a\omega_{\p, n}+b\omega_{\bar{\p}, n}, L)^{a^{m_{1}-1}b^{m_{2}-1}} \in K(p^{\infty})^{\times}/K(p^{\infty})^{\times p^{n}}.
\]

\begin{proof}[Proof of Lemma \ref{lmm:epsilon}] 
    (1) The first part of the assertion follows from Lemma \ref{lmm:theta_val} (3). Let $\mathfrak{a}$ be a nontrivial ideal of $O_{K}$ which is prime to $p$. By Lemma \ref{lmm:theta_val} (3), we have
    \begin{align*}
        \left( \frac{K(p^{\infty})/K}{\mathfrak{a}} \right)(\epsilon_{\boldsymbol{m}, n})
         & =  \prod_{\substack{0 \leq a,b<p^{n} \\ p \nmid  \gcd(a,b)}}  \theta(a\omega_{\p, n}+b\omega_{\bar{\p}, n}, \mathfrak{a}^{-1}L)^{a^{m_{1}-1}b^{m_{2}-1}} \\
         & = \prod_{\substack{0 \leq a,b<p^{n} \\ p \nmid  \gcd(a,b)}}  \theta(a\varepsilon_{\p}(\sigma_{\mathfrak{a}})\omega_{\p, n}+b\varepsilon_{\bar{\p}}(\sigma_{\mathfrak{a}})\omega_{\bar{\p}, n}, \mathfrak{a}^{-1}L)^{a^{m_{1}-1}b^{m_{2}-1}}
           \\
         & = \left( \prod_{\substack{0 \leq a,b<p^{n} \\ p \nmid  \gcd(a,b)}}  \theta(a\omega_{\p, n}+b\omega_{\bar{\p}, n}, L)^{a^{m_{1}-1}b^{m_{2}-1}} \right)^{\varepsilon_{\cyc}\varepsilon^{-\boldsymbol{m}}\left( \left( \frac{K(p^{\infty})/K}{\mathfrak{a}} \right) \right)} \\
         & =(\epsilon_{\boldsymbol{m}, n})^{\varepsilon_{\cyc}\varepsilon^{-\boldsymbol{m}}\left( \left( \frac{K(p^{\infty})/K}{\mathfrak{a}} \right) \right)},
    \end{align*} which is the desired equality. The assertion of (2) follows from Lemma \ref{lmm:theta_val} (4).
\end{proof}

\subsection{Elliptic units and cyclotomic units}\label{2.2}

In this section, we introduce a result of Kersey \cite[Theorem 4.2]{Ke80}, which explicitly compares elliptic units with cyclotomic units inside every abelian extension $H$ of $\mathbb{Q}$ containing $K$. Since we need this result when $H$ is equal to $K(\mu_{p^{n}})$ for some $n$, we only treat those cases in the following. 

In the following, let $d_{K}$ be the absolute value of the discriminant of $K$. Throughout this section, we fix an arbitrary integer $n \geq 1$. We use the mod-$p^{n}$ cyclotomic character 
\[
\varepsilon_{\cyc} \bmod p^{n} \colon
\gal(K(\mu_{p^{n}})/K) \xrightarrow{\sim} (\mathbb{Z}/p^{n}\mathbb{Z})^{\times}.
\] to identify two groups. For $c \in (\mathbb{Z}/p^{n}\mathbb{Z})^{\times}$, the corresponding element in the left-hand side will be denoted by $\sigma_{c}$. We define the (anti-equivariant) regulator map by
\[
\rho \colon K(\mu_{p^{n}})^{\times} \to \mathbb{R}[\gal(K(\mu_{p^{n}})/K)] : \; 
\alpha \mapsto
\sum_{\sigma \in \gal(K(\mu_{p^{n}})/K)}
\log | \sigma \alpha | \sigma,
\] which is often used to prove Dirichlet's unit theorem. 

\begin{thm}[cf. Kersey {\cite[Theorem 4.1]{Ke80}}]\label{thm:Kersey}
We have
\begin{align*}
& \rho \left(
\mathrm{N}_{K(p^{n})/K(\mu_{p^{n}})} \left(
\theta \left( \Omega/p^{n}, L
\right)^{p^{n}} 
\right) \right) \\
= &
-3p^{n} {\sum_{c \in (\mathbb{Z}/p^{n}\mathbb{Z})^{\times}}}
{\sum_{\substack{a,b \in \mathbb{Z}/d_{K}p^{n}\mathbb{Z} \\ ab \equiv c \bmod p^{n}}}}\ab ( \chi_{K}(a)+\chi_{K}(b) ) 
B_{1}\left( \Bab{\frac{a}{dp^{n}}} \right)
\log | 1-\zeta_{d_{K}p^{n}}^{b} |  \sigma_{c}.
\end{align*} Note that $\theta(\Omega/p^{n}, L)^{p^{n}} $ is contained in $K(p^{n})^{\times}$ by \cite[Chapter 11, Theorem 1.1]{KL81}.
\end{thm}

\begin{rmk}
Theorem 4.1 of \cite{Ke80} is stated by using the Ramachandra-Robert invariant $g_{H}(C)$, which is defined to be the image under the norm map $\mathrm{N}_{K(\mathfrak{n})/H}$ of 
\[
g_{\mathfrak{n}}(C)=\left( \mathfrak{k}(1, \mathfrak{n}\mathfrak{c}^{-1})^{12}\Delta(\mathfrak{n}\mathfrak{c}^{-1}) \right)^{N(\mathfrak{n})}=
\theta(1, \mathfrak{n}\mathfrak{c}^{-1})^{N(\mathfrak{n})}.
\] Here, $\mathfrak{n}$ is the conductor of $H/K$, $N(\mathfrak{n})$ is the smallest positive integer contained in $\mathfrak{n}$, $C$ is an element of the ray class group modulo $\mathfrak{n}$ and $\mathfrak{c}$ is an arbitrary ideal contained in $C$. Since the fundamental theta function satisfies $\theta(z, L)=\theta(cz, cL)$ for every $c \in \mathbb{C}^{\times}$ \cite[Chapter II, 2.1]{dS87}, the original theorem immediately implies the assertion of Theorem \ref{thm:Kersey}.
\end{rmk}

We also have the multiplicative version of Theorem \ref{thm:Kersey}.

\begin{thm}[cf. Kersey {\cite[Theorem 4.2]{Ke80}}]\label{thm:Kersey_mult}

For each $c \in (\mathbb{Z}/p^{n}\mathbb{Z})^{\times}$, we have
\begin{align*}
\mathrm{N}_{K(p^{n})/K(\mu_{p^{n}})}\left(
\theta \left( \Omega/p^{n}, L
\right)^{4p^{n}} 
\right)^{\sigma_{c}} 
\equiv \prod_{b \in \mathbb{Z}/d_{K}p^{n}\mathbb{Z}}(1-\zeta_{d_{K}p^{n}}^{b})^{\nu_{n}(c,b)} \bmod \mu(K(p^{n}))
\end{align*} where we define
\[
\nu_{n}(c,b) \coloneqq -12p^{n}\sum_{\substack{a \in \mathbb{Z}/d_{K}p^{n}\mathbb{Z} \\ ab \equiv c \bmod p^{n}}}\left(\chi _{K}(a)+\chi_{K}(b)\right) B_{1} \left( \Bab{\frac{a}{d_{K}p^{n}}} \right),
\] which is an integer by \cite[Corollary 5.3]{Ke80} for every $b \in \mathbb{Z}/d_{K}p^{n}\mathbb{Z}$.
\end{thm}

\begin{proof}
The assertion follows by comparing the coefficient of $\sigma_{c}$ in Theorem \ref{thm:Kersey}.
\end{proof}

\begin{cor}\label{cor:Kersey_mult}
 For a nontrivial ideal $\mathfrak{a}$ of $O_{K}$ such that $(\mathfrak{a}, p)=1$, we have
\begin{align*}
\mathrm{N}_{K(p^{n})/K(\mu_{p^{n}})} \left(
\theta_{\mathfrak{a}}(\omega_{n}) 
\right)^{4}
\equiv 
\prod_{b \in \mathbb{Z}/d_{K}p^{n}\mathbb{Z}}(1-\zeta_{d_{K}p^{n}}^{b})^{\frac{N\mathfrak{a} \cdot\nu_{n}(1,b)-\nu_{n}(N\mathfrak{a}, b)}{p^{n}}}
\bmod \mu(K(p^{n})).
\end{align*} 
\end{cor}

\begin{proof}
First, we remark that the exponent
\[
\frac{N\mathfrak{a} \cdot\nu_{n}(1,b)-\nu_{n}(N\mathfrak{a}, b)}{p^{n}}
\] is integral for every $b \in \mathbb{Z}/d_{K}p^{n}\mathbb{Z}$ by \cite[Corollary 5.3]{Ke80}. Let $(\frac{K(p^{n})/K}{\mathfrak{a}})\in \gal(K(p^{n})/K)$ be the Artin symbol corresponding to $\mathfrak{a}$. Using \cite[p. 236]{KL81} and the fact that the Artin symbol $(\frac{K(p^{n})/K}{\mathfrak{a}})$ maps to $\sigma_{N\mathfrak{a}}$ in $\gal(K(\mu_{p^{n}})/K)$, it follows that 
\begin{align*}
&\left( \frac{K(p^{n})/K}{\mathfrak{a}} \right)  \left( \theta \left(
\Omega/p^{n}, L
\right)^{p^{n}} \right)
=
\theta \left(
\Omega/p^{n}, \mathfrak{a}^{-1}L
\right)^{p^{n}}, \quad \text{and} \\
&\mathrm{N}_{K(p^{n})/K(\mu_{p^{n}})}  \left(
\theta \left ( \Omega/p^{n}, \mathfrak{a}^{-1}L
\right)^{p^{n}} \right)
=
\mathrm{N}_{K(p^{n})/K(\mu_{p^{n}})}  \left( 
\theta \left (\Omega/p^{n}, L \right)^{p^{n}} \right)^{\sigma_{N\mathfrak{a}}}.
\end{align*} The assertion follows by applying Theorem \ref{thm:Kersey_mult} to 
\begin{align*}
\mathrm{N}_{K(p^{n})/K(\mu_{p^{n}})} \left(
\theta_{\mathfrak{a}}(\omega_{n})
\right)^{4p^{n}} 
&= \mathrm{N}_{K(p^{n})/K(\mu_{p^{n}})} \left( \frac{
\theta \left (\Omega/p^{n}, L
\right)^{4p^{n}N\mathfrak{a}}}
{
\theta \left ( \Omega/p^{n}, \mathfrak{a}^{-1}L
\right)^{4p^{n}} 
} \right) \\
&=
\mathrm{N}_{K(p^{n})/K(\mu_{p^{n}})}  \left(
\theta \left ( \Omega/p^{n}, L
\right)^{4p^{n}} \right)^{N\mathfrak{a}-\sigma_{N\mathfrak{a}}}.
\end{align*}
\end{proof}

\subsection{\'Etale regulator map and Soul\'e cocycles}\label{2.3}

In this section, we introduce a certain analogue $\rho_{\et}$ of the regulator map $\rho$ introduced in the last section. The construction of $\rho_{\et}$ seems to originate from Soul\'e \cite{So81} using algebraic $K$-theory, particularly his \'etale chern class map introduced in \cite{So79}. 

In the rest of this paper, we denote 
$
\boldsymbol{\zeta} \coloneqq (\zeta_{p^{n}})_{n} \in \varprojlim \mu_{p^{n}}
$ and often identify $\varprojlim \mu_{p^{n}}$ with the Tate twist $\mathbb{Z}_{p}(1)$ through this basis. We define the (completed) group rings by 
\[
\Lambda_{n} \coloneqq \mathbb{Z}_{p} \lbrack \gal(K(\mu_{p^{n}})/K) \rbrack
\quad \text{for $n \geq 1$, } \quad
\Lambda \coloneqq \mathbb{Z}_{p}\lbrack \lbrack \gal(K(\mu_{p^{\infty}})/K) \rbrack \rbrack = \varprojlim \Lambda_{n},
\] and denote the ring of $p$-integers in $K(\mu_{p^{n}})$ by $O_{n}$. According to this notation, we have the composite of the following two isomorphisms:
\[
\rho_{\et,n} \colon
O_{n}^{\times}\otimes \mathbb{Z}_{p}
\xrightarrow{\sim}
H^{1}_{\et}(O_{n}, \mathbb{Z}_{p}(1))
\xrightarrow[\mathrm{sh}^{-1}]{\sim}
H^{1}_{\et}(O_{K}[1/p], \Lambda_{n}(1))
\] where the first map is the Kummer isomorphism and the second one is the inverse of the isomorphism given by Shapiro's Lemma in \cite[(1.6.4) Proposition]{NSW08}. Here, we identify the induced module $\mathrm{Ind}^{G_{K(\mu_{p^{n}})}}_{G_{K}}(\mathbb{Z}_{p}(1))=\mathrm{Map}(\gal(K(\mu_{p^{n}})/K), \mathbb{Z}_{p}(1))$ \cite[p. 62]{NSW08} with $\Lambda_{n}(1)$ through the equivariant isomorphism
\[
\mathrm{Ind}^{G_{K(\mu_{p^{n}})}}_{G_{K}}(\mathbb{Z}_{p}(1)) \xrightarrow{\sim} \Lambda_{n}(1)=\mathbb{Z}_{p}[\gal(K(\mu_{p^{n}})/K)]  \otimes \mathbb{Z}_{p}(1) : \; x \mapsto \sum_{\tau \in \gal(K(\mu_{p^{n}})/K)}\tau \otimes  \tau x(\tau^{-1}).
\]
We remark that the map 
\[
H^{1}_{\et}(O_{n}, \mathbb{Z}_{p}(1))
\xrightarrow[\mathrm{sh}^{-1}]{\sim}
H^{1}_{\et}(O_{K}[1/p], \Lambda_{n}(1))
\] is an isomorphism between $\gal(K(\mu_{p^{n}})/K)$-modules:  The group $\gal(K(\mu_{p^{n}})/K)$ acts on the cohomology group on the left-hand side by conjugation. On the other hand, for $\sigma \in \gal(K(\mu_{p^{n}})/K)$, the homomorphism
\[
\sigma^{-1}: \Lambda_{n}(1) \xrightarrow{\sim} \Lambda_{n}(1) : \tau \otimes \zeta_{p^{n}} \mapsto \sigma^{-1}\tau \otimes \zeta_{p^{n}}
\] is an isomorphism between $\mathbb{Z}_{p}[\gal(K(\mu_{p^{n}})/K)]$-modules, and this defines an action of the group $\gal(K(\mu_{p^{\infty}})/K)$ on the cohomology group on the right-hand side. The desired equivariance follows from the discussion after \cite[(1.6.5) Proposition]{NSW08}.

\begin{lmm}\label{lmm:reg}
For each $n \geq 1$, the following assertions hold. 
\begin{enumerate}
\item For $\alpha \in O_{n}^{\times}$ and $\sigma \in \gal(K(\mu_{p^{n}})/K)$, we have
\[
\rho_{\et,n}(\alpha^{\sigma})=\sigma^{-1}\rho_{\et, n}(\alpha).
\] 
\item The following diagram commutes:
\begin{eqnarray*}
\begin{tikzcd}
O_{n+1}^{\times} \otimes \mathbb{Z}_{p} \arrow[r, "\rho_{\et,n+1}"]\arrow[d, "\mathrm{norm}"] & H^{1}_{\et}(O_{K}[1/p], \Lambda_{n+1}(1)) \arrow[d] \\
O_{n}^{\times} \otimes \mathbb{Z}_{p} \arrow[r, "\rho_{\et, n}"] & H^{1}_{\et}(O_{K}[1/p], \Lambda_{n}(1))
\end{tikzcd}
\end{eqnarray*}
where the right vertical arrow is induced by the natural surjection $\Lambda_{n+1} \to \Lambda_{n}$.
\end{enumerate}
\end{lmm}   

\begin{proof}
The assertion (1) follows from the discussion before this lemma and the equivariance of the Kummer isomorphism. The assertion (2) follows from \cite[(1.6.5) Proposition]{NSW08}.
\end{proof}

\begin{dfn}[The regulator map $\rho_{\et}$]
We define the regulator map $\rho_{\et}$ by
\[
\rho_{\et} \coloneqq (\rho_{\et,n})_{n} \colon 
\varprojlim_{n} O_{n}^{\times} \otimes \mathbb{Z}_{p}  \xrightarrow{\sim}
H^{1}_{\et}(O_{K}[1/p], \Lambda(1))=\varprojlim_{n}H^{1}_{\et}(O_{K}[1/p], \Lambda_{n}(1)).
\] 
\end{dfn}

Note that $H^{1}_{\et}(O_{K}[1/p], \Lambda(1))$ can be identified with the first continuous cochain cohomology group (with coefficients in $ \Lambda(1)$) of the Galois group of the maximal Galois extension of $K$ unramified outside $p$. 

\begin{rmk}\label{rmk:equiv}
By Lemma \ref{lmm:reg} (1), the map $\rho_{\et}$ satisfies the anti-equivariance relation \[
\rho_{\et} \left( \boldsymbol{\alpha}^{\sigma} \right)
=
\sigma^{-1} \rho_{\et} \left( \boldsymbol{\alpha} \right)
\] for every norm-compatible system of units $\boldsymbol{\alpha}$ and $\sigma \in \gal(K(\mu_{p^{\infty}})/K)$. This property is similar to that of the regulator map $\rho$ introduced in the last section.
\end{rmk}

The map $\rho_{\et}$ allows us to regard norm-compatible systems of units as $\Lambda(1)$-valued cocycle classes. For a continuous character $\psi \colon \mathrm{Gal}(K(\mu_{p^{\infty}})/K) \to \mathbb{Z}_{p}^{\times}$, we denote the two associated evaluation maps $\Lambda \to  \mathbb{Z}_{p}(\psi)$ and $H^{1}_{\et}(O_{K}[1/p], \Lambda(1)) \to H^{1}_{\et}(O_{K}[1/p], \mathbb{Z}_{p}(\psi)(1))$ by the same letter. We define the homomorphism $\rho_{\et,\psi}$ to be the composite of $\rho_{\et}$ and the evaluation map at $\psi \varepsilon_{\cyc}^{-1}$:  
\[
\rho_{\et, \psi} \colon 
\varprojlim_{n} O_{n}^{\times} \otimes \mathbb{Z}_{p}
\xrightarrow{\rho_{\et}} 
H^{1}_{\et}(O_{K}[1/p], \Lambda(1)) 
\xrightarrow{\psi \varepsilon_{\cyc}^{-1}}
H^{1}_{\et}(O_{K}[1/p], \mathbb{Z}_{p}(\psi)). 
\]

Let $\boldsymbol{\alpha}=(\alpha_{n})$ be a norm-compatible system of units, and consider its image under the composite of the map $\rho_{\et, \psi}$ and the restriction map to
\[
H^{1}_{\et}(O_{K(\mu_{p^{\infty}})}[1/p], \mathbb{Z}_{p}(\psi))^{\gal(K(\mu_{p^{\infty}})/K)}=\mathrm{Hom}_{\gal(K(\mu_{p^{\infty}})/K)}(\pi_{1}^{\et}(O_{K(\mu_{p^{\infty}})}[1/p]), \mathbb{Z}_{p}(\psi)).
\] Here, the \'etale fundamental group $\pi_{1}^{\et}(O_{K(\mu_{p^{\infty}})}[1/p])$ is identified with the Galois group of the maximal Galois extension of $K(\mu_{p^{\infty}})$ unramified outside $p$. By construction, the Kummer character $\rho_{\et, \psi}(\boldsymbol{\alpha})$ is characterized by
\[
\zeta_{p^{n}}^{\rho_{\et, \psi}(\boldsymbol{\alpha})(\sigma)}=\frac{\sigma \left( \prod_{a \in (\mathbb{Z}/p^{n}\mathbb{Z})^{\times}}(\sigma_{a}\alpha_{n})^{\frac{\psi \varepsilon_{\cyc}^{-1}(\sigma_{a})}{p^{n}}} \right)}{\prod_{a \in (\mathbb{Z}/p^{n}\mathbb{Z})^{\times}}(\sigma_{a}\alpha_{n})^{\frac{\psi \varepsilon_{\cyc}^{-1}(\sigma_{a})}{p^{n}}}} \quad \text{for $\sigma \in G_{K(\mu_{p^{\infty}})}$ and $n \geq 1$.}
\]

\begin{ex}[Soul\'e cocycle classes]\label{ex:Soule}
    From the norm-compatible system
    \[
    \boldsymbol{c}^{(1)} \coloneqq (1-\zeta_{p^{n}})_{n} \in 
\varprojlim_{n} O_{n}^{\times} \otimes \mathbb{Z}_{p}
    \] we obtain the cocycle class $\rho_{\et}(\boldsymbol{c}^{(1)})$. For an integer $m \geq 2$, we call $
    \chi_{m} \coloneqq \rho_{\et, \varepsilon_{\cyc}^{m}}(\boldsymbol{c}^{(1)})
    $ the $m$-th Soul\'e cocycle class. By its construction, this extends the $m$-th Soul\'e character on $G_{K(\mu_{p^{\infty}})}$ introduced in Definition \ref{dfn:Soule}. 
    
    Note that this cocycle is torsion whenever $m$ is even: More generally, we claim that $\rho_{\et, \psi}(\boldsymbol{c}^{(1)})$ is torsion whenever $\psi$ is nontrivial and even (i.e. $\psi(\sigma_{-1})=1$).  In fact, since $\psi$ is even, we have
    \[
    \rho_{\et, \psi} (\sigma_{-1}\boldsymbol{c}^{(1)})
    =
    -\rho_{\et, \psi} ( \boldsymbol{c}^{(1)} )
    \] and $
    \sigma_{-1}\boldsymbol{c}^{(1)}=(-\boldsymbol{\zeta}) \cdot \boldsymbol{c}^{(1)}$. Hence the claim is reduced to showing that $\rho_{\et, \psi}(\boldsymbol{\zeta})$ is torsion, which is proved in the following lemma:
\end{ex}

\begin{lmm}\label{lmm:unity}
    The following two assertions hold:
    \begin{enumerate}
        \item The homomorphism
        \[
        \varprojlim_{n} \mu(K(\mu_{p^{n}})) \otimes \mathbb{Z}_{p}
    \to
     \varprojlim_{n} O_{n}^{\times} \otimes \mathbb{Z}_{p}^{\times}
    \xrightarrow{\rho_{\et}}
    H^{1}_{\et}(O_{K}[1/p], \Lambda(1)) 
        \] factors through the composite of an isomorphism
        \[
         \varprojlim_{n} \mu(K(\mu_{p^{n}})) \otimes \mathbb{Z}_{p} \cong H^{1}(\gal(K(\mu_{p^{\infty}})/K), \Lambda(1))
        \] and the inflation map.
        \item Let $\psi \colon \mathrm{Gal}(K(\mu_{p^{\infty}})/K) \to \mathbb{Z}_{p}^{\times}$ be a nontrivial character. Then the image of the homomorphism \[
    \varprojlim_{n} \mu(K(\mu_{p^{n}})) \otimes \mathbb{Z}_{p}
    \to
     \varprojlim_{n} O_{n}^{\times} \otimes \mathbb{Z}_{p}^{\times}
    \xrightarrow{\rho_{\et,\psi}}
    H^{1}_{\et}(O_{K}[1/p], \mathbb{Z}_{p}(\psi)) 
    \] is a torsion subgroup.
    \end{enumerate}
\end{lmm}
\begin{proof}
    First, observe that the image of $\zeta_{n}$ under the Kummer isomorphism generates the image of the inflation map
    \[
    \mathbb{Z}/p^{n}\mathbb{Z}  \cong H^{1}(\gal(K(\mu_{p^{\infty}})/K(\mu_{p^{n}})), \mathbb{Z}_{p}(1)) \hookrightarrow  H^{1}_{\et}(O_{n}[1/p], \mathbb{Z}_{p}(1))
    \] for every $n \geq 1$. By using Shapiro's lemma and taking the limit of both sides, it follows that $\rho_{\et}(\boldsymbol{\zeta})$ generates the image of the inflation map
    \[
     \mathbb{Z}_{p} \cong H^{1}(\gal(K(\mu_{p^{\infty}})/K), \Lambda(1)) \hookrightarrow  H^{1}_{\et}(O_{K}[1/p], \Lambda(1)),
    \] proving the first assertion. Now $\rho_{\et, \psi}(\boldsymbol{\zeta})$ is contained in the image of 
    \[
     H^{1}(\gal(K(\mu_{p^{\infty}})/K), \mathbb{Z}_{p}(\psi)) \cong \mathbb{Z}_{p}(\psi)_{\gal(K(\mu_{p^{\infty}})/K)},
    \] where $\mathbb{Z}_{p}(\psi)_{\gal(K(\mu_{p^{\infty}})/K)}$ denotes the $\gal(K(\mu_{p^{\infty}})/K)$-coinvariant of $\mathbb{Z}_{p}(\psi)$. This module is finite since $\psi$ is nontrivial, concluding the proof of (2).
\end{proof}

\begin{rmk}
    When $\psi$ is trivial, the proof of Lemma \ref{lmm:unity} shows that the image of $\varprojlim_{n} \mu(K(\mu_{p^{n}})) \otimes \mathbb{Z}_{p}$ in $H^{1}(O_{K}[1/p], \mathbb{Z}_{p}) \cong \mathbb{Z}_{p}^{2}$ coincides with a $\mathbb{Z}_{p}$-submodule generated by the $p$-adic cyclotomic character $\varepsilon_{\cyc}$. 
\end{rmk}

We will also need the following variant of the Soul\'e cocycles:

\begin{dfn}\label{dfn:genSoule}
    For $m \geq 2$, we define the cocycle class
    \[
    \chi_{m}^{(d_{K})} \coloneqq
    \rho_{\et, \varepsilon_{\cyc}^{m}}(\boldsymbol{c}^{(d_{K})})
    \in
    H^{1}_{\et}(O_{K}[1/p], \mathbb{Z}_{p}(m)) 
    \] where we define the norm-compatible system by
    \[
    \boldsymbol{c}^{(d_{K})} \coloneqq 
    \left(
    \mathrm{N}_{\mathbb{Q}(\mu_{d_{K}p^{n}})/K(\mu_{p^{n}})}(1-\zeta_{d_{K}p^{n}})
    \right)_{n}.
    \]
\end{dfn}

\begin{rmk}\label{rmk:genSoule}
\begin{enumerate} 
\item The class $\chi_{m}^{(d_{K})}$ corresponds to the Soul\'e-Deligne \emph{cyclotomic element} $c_{m}(\zeta_{d_{K}})$ in Huber-Kings \cite[Definition 3.1.2]{HK03} via the corestriction map. Similarly, the $m$-th Soul\'e cocycle corresponds to $c_{m}(1)$ in their paper. We refer to \cite[Section 3]{HK03} for a detailed discussion on cyclotomic elements.

\item When restricted to $G_{K(\mu_{p^{\infty}})}$, the character $\chi_{m}^{(d_{K})}$ is uniquely characterized by 
\[
\zeta_{p^{n}}^{\chi_{m}^{(d_{K})}(\sigma)}=\frac{\sigma \left( \prod_{a \in (\mathbb{Z}/p^{n}\mathbb{Z})^{\times}}\sigma_{a}\mathrm{N}_{\mathbb{Q}(\mu_{d_{K}p^{n}})/K(\mu_{p^{n}})}(1-\zeta_{d_{K}p^{n}})^{\frac{a^{m-1}}{p^{n}}} \right)}{\prod_{a \in (\mathbb{Z}/p^{n}\mathbb{Z})^{\times}}\sigma_{a}\mathrm{N}_{\mathbb{Q}(\mu_{d_{K}p^{n}})/K(\mu_{p^{n}})}(1-\zeta_{d_{K}p^{n}})^{\frac{a^{m-1}}{p^{n}}}}
\] for every $n \geq 1$.
\item Let us recall \cite[Corollary of Proposition 4]{IS87} that $\chi_{m}$ is surjective as a character on $G_{K(\mu_{p^{\infty}})}$ for every $m=3,5,\dots,p-2$ if and only if Vandiver's conjecture holds for $p$. Similarly, the same argument as the proof of \cite[Corollary]{IS87} shows that the following two assertions are equivalent:
    \begin{enumerate}
    \item $\chi_{m}$ is surjective for $m=3,5, \dots, p-2$, and $\chi_{m}^{(d_{K})}$ is surjective for $m=2,4, \dots, p-3$.
    \item The class number of $K(\mu_{p})^{+}$ is not divisible by $p$, where $K(\mu_{p})^{+}$ denotes the maximal real subfield of $K(\mu_{p})$.
    \end{enumerate}
\end{enumerate}
\end{rmk}

\begin{lmm}\label{lmm:key}
Let $\iota \in \gal(\mathbb{Q}(\mu_{d_{K}p^{\infty}})/\mathbb{Q}(\mu_{p^{\infty}}))$ be a unique element satisfying $\iota(\zeta_{d_{K}})=\zeta_{d_{K}}^{-1}$ and $\varphi$ Euler's totient function. Then the following two equalities hold:
\begin{align*}
& \rho_{\et}(\boldsymbol{c}^{(d_{K})})+\rho_{\et}(\iota \boldsymbol{c}^{(d_{K})})=\prod_{\substack{\ell: \mathrm{prime} \\ \ell \mid d_{K}}}(1-\sigma_{\ell}) \cdot \rho_{\et}(\boldsymbol{c}^{(1)}), \\
& \rho_{\et}(\boldsymbol{c}^{(d_{K})})-\sigma_{-1} \cdot \rho_{\et}(\iota \boldsymbol{c}^{(d_{K})})=\frac{\varphi(d_{K})}{2d_{K}}\rho_{\et}(\boldsymbol{\zeta}).
\end{align*} 
\end{lmm}
\begin{proof}
The first equality follows from
\[
\boldsymbol{c}^{(d_{K})} \cdot \iota \boldsymbol{c}^{(d_{K})}
=
\left( 
\mathrm{N}_{\mathbb{Q}(\mu_{d_{K}p^{n}})/\mathbb{Q}(\mu_{p^{n}})}(1-\zeta_{d_{K}p^{n}})
\right)_{n}
=
\left(
(1-\zeta_{p^{n}})^{\prod_{\ \ell \mid d_{K}}(1-\sigma_{\ell}^{-1})}
\right)_{n}.
\] In the following of the proof, we denote an element of $\gal(\mathbb{Q}(\mu_{d_{K}p^{\infty}})/\mathbb{Q}(\mu_{d_{K}}))$ corresponding to $\sigma_{-1}$ under the isomorphism $
 \gal(\mathbb{Q}(\mu_{d_{K}p^{\infty}})/\mathbb{Q}(\mu_{d_{K}})) \xrightarrow{\sim}  \gal(K(\mu_{p^{\infty}})/K)
$ by the same letter $\sigma_{-1}$. Then we have $
\rho_{\et}(\sigma_{-1}\iota \boldsymbol{c}^{(d_{K})})
=
\sigma_{-1} \cdot \rho_{\et}(\iota \boldsymbol{c}^{(d_{K})})
$, and 
\begin{align*}
\sigma_{-1}\iota \boldsymbol{c}^{(d_{K})}
= &
\left( \mathrm{N}_{\mathbb{Q}(\mu_{d_{K}p^{n}})/K(\mu_{p^{n}})}(1-\zeta_{d_{K}p^{n}}^{-1})
\right)_{n} \\
= &
\left( \mathrm{N}_{\mathbb{Q}(\mu_{d_{K}p^{n}})/K(\mu_{p^{n}})}(-\zeta_{d_{K}p^{n}}^{-1}) \right)_{n} \cdot \boldsymbol{c}^{(d_{K})} \\
= &
\left( -\boldsymbol{\zeta}\right)^{-\frac{\varphi(d_{K})}{2d_{K}}} \cdot \boldsymbol{c}^{(d_{K})}.
\end{align*} Moreover, we have $\rho_{\et}(-\boldsymbol{\zeta})=\rho_{\et}(\boldsymbol{\zeta})$ since $p$ is odd. Hence the second equality follows.
\end{proof}

Recall that, by Lemma \ref{lmm:unity} (1), the element $\rho_{\et}(\boldsymbol{\zeta})$ generates the image of the inflation map 
\[
H^{1}(\gal(K(\mu_{p^{\infty}})/K), \Lambda(1)) \hookrightarrow H^{1}_{\et}(O_{K}[1/p], \Lambda(1)).
\] Hence we obtain the following corollary.

\begin{cor}\label{cor:key}Keeping the notation of the previous lemma, the following equality holds in $H^{1}_{\et}(O_{K}[1/p], \Lambda(1))/H^{1}(\gal(K(\mu_{p^{\infty}})/K), \Lambda(1))$:
\[
\chi_{m}^{(d_{K})}
=
\begin{dcases}
\text{$-\rho_{\et, \varepsilon_{\cyc}^{m}}(\iota \boldsymbol{c}^{(d_{k})})$ (and $\chi_{m}=0$ by Example \ref{ex:Soule})}  & \text{if $m \geq 2$ is even,} \\
\rho_{\et, \varepsilon_{\cyc}^{m}}(\iota \boldsymbol{c}^{(d_{k})})=\frac{\prod_{\ell \mid d_{K}}(1-\ell^{m-1})}{2} \chi_{m} & \text{if $m \geq 3$ is odd.}
\end{dcases}
\]
\end{cor}
\begin{proof}
    The assertion follows from Lemmas \ref{lmm:unity} and \ref{lmm:key}.
\end{proof}

\section{Properties of elliptic Soul\'e characters}\label{3}

\subsection{Proof of the factorization formula}\label{3.1}

In this section, we give a proof of Theorem \ref{thm:main}. First, we briefly recall a computation of the first author \cite{Is23}, relating the elliptic Soul\'e character $\chi_{\boldsymbol{m}}^{\E}$ to a certain cocycle $\chi_{m, \mathfrak{a}}^{\E}$ constructed via $\rho_{\et, \varepsilon_{\cyc}^{m}}$. 

\begin{dfn}[cf. {\cite[Definition 4.2]{Is23}}]
    Let $\mathfrak{a}$ be a nontrivial ideal of $O_{K}$ such that $(\mathfrak{a},p)=1$, and note that $(\theta_{\mathfrak{a}}(\omega_{n}))_{n \geq 1}$ forms a norm-compatible system of units by \cite[Chapter II, 2.4 Proposition (ii)]{dS87}. For an integer $m \geq 2$, we define \[
    \chi_{m, \mathfrak{a}}^{\E} \coloneqq \rho_{\et, \varepsilon_{\cyc}^{m}}\left( \left( \mathrm{N}_{K(p^{n})/K(\mu_{p^{n}})}\theta_{\mathfrak{a}}(\omega_{n}) \right)_{n} \right) \in H^{1}_{\et}(O_{K}[1/p], \mathbb{Z}_{p}(m)).
    \] 
\end{dfn}

\begin{rmk}
In fact, in \cite[Definition 4.2]{Is23}, the cocycle $\chi_{m, \mathfrak{a}}^{\E}$ is defined as an image of the norm-compatible unit $\left( \theta_{\mathfrak{a}}(\omega_{n}) \right)_{n}$ under a certain homomorphism
\[ 
\varprojlim_{n} H^{0}_{\et}(O_{K(p^{n})}\fracp, \mathbb{Z}_{p}(1)) \to H^{1}_{\et}(O_{K}\fracp, \mathbb{Z}_{p}(m)),
\] whose definition is similar to that of $\rho_{\et, \varepsilon_{\cyc}^{m}}$. This map factors through the corestriction map 
\[
\varprojlim_{n} H^{0}_{\et}(O_{K(p^{n})}\fracp, \mathbb{Z}_{p}(1)) 
\to
\varprojlim_{n} H^{0}_{\et}(O_{K(\mu_{p^{n}})}\fracp, \mathbb{Z}_{p}(1))
\] induced by $\mathrm{Spec}(O_{K(p^{n})}[1/p]) \to \mathrm{Spec}(O_{K(\mu_{p^{n}})}[1/p])$, and the resulting map coincides with $\rho_{\et, \varepsilon_{\cyc}^{m}}$.  
\end{rmk}

\begin{prp}[cf. {\cite[Proposition 4.3]{Is23}}]\label{prp:comp}
For a nontrivial ideal $\mathfrak{a}$ of $O_{K}$ prime to $p$ and $m \geq 2$, we have the following equality between characters on $G_{K(p^{\infty})}$ :
\[
w_{K} \frac{\chi_{m, \mathfrak{a}}^{\E}}{(1-p^{m-1})^{2}}
=
\left(
  N\mathfrak{a}-N\mathfrak{a}^{1-m}
\right)
\frac{\chi_{m}^{\E}}{1-p^{2m-2}}.
\] 
\end{prp}

 Results obtained in the previous section allow us to compare three characters $\chi_{m, \mathfrak{a}}^{\E}$, $\chi_{m}$ and $\chi_{m}^{(d_{K})}$. Before explaining the result, we recall the notion of the Stickelberger element \cite[\textsection 6.2]{Wa82}. For an integer $d$ such that $(d,p)=1$, we define 
\[
\boldsymbol{\theta}_{d}=(\theta_{d,n}) \coloneqq  \ab (\sum_{a \in (\mathbb{Z}/dp^n\mathbb{Z})^\times} B_1 \ab( \Bab{\frac{a}{dp^{n}}}) \sigma_a^{-1} )_n \in \varprojlim_{n} \mathbb{Q}[\gal(\mathbb{Q}(\mu_{dp^{n}})/\mathbb{Q})].
\] For every integer $c$ such that $(c,pd)=1$, it follows by \cite[Lemma 6.9]{Wa82} that
\[
(1-c\sigma_{c}^{-1}) \boldsymbol{\theta}_{d} \in \mathbb{Z}_{p}[[\gal(\mathbb{Q}(\mu_{dp^{\infty}})/\mathbb{Q})]].
\] We denote the involution on $\mathbb{Z}_{p}[[\gal(\mathbb{Q}(\mu_{dp^{\infty}})/\mathbb{Q})]]$ induced by $\sigma_{a} \mapsto \sigma_{a}^{-1}$ by $\ast$. Then the image of the Stickelberger element under this involution
\[
(1-c\sigma_{c}) \boldsymbol{\theta}_{d}^{\ast} \in \mathbb{Z}_{p}[[\gal(\mathbb{Q}(\mu_{dp^{\infty}})/\mathbb{Q})]]
\] can be described in terms of the Kubota-Leopoldt $p$-adic $L$-function by e.g. \cite[Examples (2) before Theorem 12.4]{Wa82}. In the following theorem, we will regard $(1-c\sigma_{c})\boldsymbol{\theta}_{1}^{\ast}$ as an element of $\Lambda=\mathbb{Z}_{p}[[\gal(K(\mu_{p^{\infty}})/K)]]$. We also regard $(1-c\sigma_{c}^{-1})\chi_{K}(\boldsymbol{\theta}_{d_{K}}^{\ast})$ as an element of $\mathbb{Z}_{p}[[\gal(K(\mu_{p^{\infty}})/K)]]$, where $\chi_{K}$ denotes a homomorphism \[
\mathbb{Z}_{p}[[\gal(\mathbb{Q}(\mu_{d_{K}p^{\infty}})/\mathbb{Q})]] 
\to
\mathbb{Z}_{p}[[\gal(\mathbb{Q}(\mu_{p^{\infty}})/\mathbb{Q})]]\cong \Lambda 
\] induced by the Dirichlet character $\chi_{K}$ associated to $K$. 

\begin{thm}\label{thm:main_measure}
Let $\mathfrak{a}$ be a nontrivial ideal of $O_{K}$ such that $(\mathfrak{a},d_{K}p)=1$. The following equality holds in $H^{1}_{\et}(O_{K}[1/p], \Lambda(1))/H^{1}(\gal(K(\mu_{p^{\infty}})/K), \Lambda(1))$:
\begin{align*}
& \rho_{\et}\left( \left( \mathrm{N}_{K(p^{n})/K(\mu_{p^{n}})}\theta_{\mathfrak{a}}(\omega_{n}) \right)_{n} \right) \\
=&
-3\left( (N\mathfrak{a}-\sigma_{N\mathfrak{a}}^{-1}) 
\chi_{K}(\boldsymbol{\theta}_{d_{K}}^{\ast}) \rho_{\et}(\boldsymbol{c}^{(1)})
+(1-\sigma_{-1})(N\mathfrak{a}-\sigma_{N\mathfrak{a}}^{-1}) 
\boldsymbol{\theta}_{1}^{\ast} \rho_{\et}(\boldsymbol{c}^{(d_{K})})
\right).
\end{align*}
\end{thm}

\begin{proof}
First, we construct $\alpha_{n}, \beta_{n}^{+}$ and $\beta_{n}^{-} \in \mathbb{Z}[\gal(K(\mu_{p^{n}})/K)]$ such that the equality
\[
\mathrm{N}_{K(p^{n})/K(\mu_{p^{n}})}(\theta_{\mathfrak{a}}(\omega_{n}))^{4} \equiv (1-\zeta_{p^{n}})^{\alpha_{n}} \cdot \mathrm{N}(1-\zeta_{d_{K}p^{n}})^{\beta_{n}^{+}} \cdot \mathrm{N}(1-\iota(\zeta_{d_{K}p^{n}}))^{\beta_{n}^{-}} \bmod \mu(K(\mu_{p^{n}}))
\] holds for each $n$, where $\iota$ is defined in Lemma \ref{lmm:key}. Using Corollary \ref{cor:Kersey_mult} and the observation that the exponent $\frac{N\mathfrak{a} \cdot\nu_{n}(1,b)-\nu_{n}(N\mathfrak{a}, b)}{p^{n}}$ in that corollary depends only on the pair $(\chi_{K}(b), b \bmod p^{n})$, we can rewrite $\mathrm{N}_{K(p^{n})/K(\mu_{p^{n}})}(\theta_{\mathfrak{a}}(\omega_{n}))^{4}$ as 
\[
\mathrm{N}_{K(p^{n})/K(\mu_{p^{n}})}(\theta_{\mathfrak{a}}(\omega_{n}))^{4} \equiv 
\prod_{(\alpha,\beta) \in \{ 0, \pm 1\} \times (\mathbb{Z}/p^{n}\mathbb{Z})^{\times}} \left(\prod_{(\chi_{K}(b),b \bmod p^{n})=(\alpha, \beta)}(1-\zeta_{d_{K}p^{n}}^{b}) \right)^{\frac{N\mathfrak{a} \cdot\nu_{n}(1,b)-\nu_{n}(N\mathfrak{a}, b)}{p^{n}}}
\] modulo $\mu(K(\mu_{p^{n}}))$. In the following, we simplify the term
\[
\prod_{(\chi_{K}(b),b \bmod p^{n})=(\alpha, \beta)}(1-\zeta_{d_{K}p^{n}}^{b})
\] according to $(\alpha, \beta)$. If $\alpha=1$, then we have
\begin{align*}
\prod_{(\chi_{K}(b),b \bmod p^{n})=(1, \beta)}(1-\zeta_{d_{K}p^{n}}^{b})
&= \prod_{b} (1-\zeta_{d_{K}}^{bp^{-n}} \cdot \zeta_{p^{n}}^{d_{K}^{-1}b}) \\
&= \prod_{\substack{a \in \mathbb{Z}/d_{K}\mathbb{Z} \\ \chi_{K}(a)=1}} (1-\zeta_{d_{K}}^{a} \cdot \zeta_{p^{n}}^{d_{K}^{-1}b}) \\
&= \mathrm{N}_{\mathbb{Q}(\mu_{d_{K}p^{n}})/K(\mu_{p^{n}})}(1-\zeta_{d_{K}p^{n}})^{\sigma_{b}}.    
\end{align*} Here, we use $\chi_{K}(p)=1$ to deduce the second equality. A similar argument shows that
\[
\prod_{(\chi_{K}(b),b \bmod p^{n})=(-1, \beta)}(1-\zeta_{d_{K}p^{n}}^{b})=\mathrm{N}_{\mathbb{Q}(\mu_{d_{K}p^{n}})/K(\mu_{p^{n}})}(1-\iota(\zeta_{d_{K}p^{n}}))^{\sigma_{b}}
\] when $\alpha=-1$. Finally, if $\alpha=0$, then we have
\begin{align*}
    \prod_{(\chi_{K}(b),b \bmod p^{n})=(0, \beta)}(1-\zeta_{d_{K}p^{n}}^{b})
&= \prod_{\zeta} (1-\zeta \cdot \zeta_{p^{n}}^{d_{K}^{-1}b})
\end{align*} where $\zeta$ runs over the set of non-primitive $d_{K}$-th roots of unity. Note that $d_{K}$ is a power of a prime $\ell$ since we assume that the class number of $K$ is one. Using this observation, it holds by direct computation that
\[
 \prod_{\zeta} (1-\zeta \cdot \zeta_{p^{n}}^{d_{K}^{-1}b})=(1-\zeta_{p^{n}}^{\ell^{-1}b})=(1-\zeta_{p^{n}})^{\sigma_{\ell^{-1}b}}.
\] Hence we have
\[
\mathrm{N}_{K(p^{n})/K(\mu_{p^{n}})}\left( \theta_{\mathfrak{a}}(\omega_{n}) \right)^{4} \equiv (1-\zeta_{p^{n}})^{\alpha_{n}} \cdot \mathrm{N}(1-\zeta_{d_{K}p^{n}})^{\beta_{n}^{+}} \cdot \mathrm{N}(1-\iota(\zeta_{d_{K}p^{n}}))^{\beta_{n}^{-}} \bmod \mu(K(\mu_{p^{n}})),
\] where we define three elements $\alpha_{n}$, $\beta_{n}^{+}$ and $\beta_{n}^{-}$ by
\begin{align*}
    \alpha_{n} &=  -12 \sigma_{\ell}^{-1}\sum_{b \in (\mathbb{Z}/p^{n}\mathbb{Z})^{\times}}\sum_{\substack{a \in (\mathbb{Z}/d_{K}p^{n}\mathbb{Z})^{\times} \\ ab \equiv 1 \bmod p^{n}}}\chi_{K}(a)\left( N\mathfrak{a}B_{1} \left( \Bab{\frac{a}{d_{K}p^{n}}} \right)-B_{1} \left( \Bab{\frac{aN\mathfrak{a}}{d_{K}p^{n}}} \right) \right) \sigma_{b},\\
    \beta_{n}^{\pm} &= -12 \sum_{b \in (\mathbb{Z}/p^{n}\mathbb{Z})^{\times}}\sum_{\substack{a \in \mathbb{Z}/d_{K}p^{n}\mathbb{Z} \\ ab \equiv 1 \bmod p^{n}}}\left(\chi_{K}(a) \pm 1 \right)\left( N\mathfrak{a}B_{1} \left( \Bab{\frac{a}{d_{K}p^{n}}} \right)-B_{1} \left( \Bab{\frac{aN\mathfrak{a}}{d_{K}p^{n}}} \right) \right) \sigma_{b}. 
\end{align*} Here, to obtain these expressions from the definition of $\nu_{n}(c,b)$ given in Theorem \ref{thm:Kersey_mult}, we use $\chi_{K}(N\mathfrak{a})=1$. In terms of Stickelberger elements, these three elements are also written as
\begin{align*}
    \alpha_{n} &= -12\sigma_{\ell}^{-1}(N\mathfrak{a}-\sigma_{N\mathfrak{a}})\chi_{K}(\theta_{d_{K},n}),\\
    \beta_{n}^{\pm} &=  -12(N\mathfrak{a}-\sigma_{N\mathfrak{a}})\left(\chi_{K}(\theta_{d_{K},n}) \pm \theta_{1,n} \right).
\end{align*} 
Hence it follows by Lemma \ref{lmm:key} that
\begin{align*}
    & \rho_{\et}\left( \left( \mathrm{N}_{K(p^{n})/K(\mu_{p^{n}})}\theta_{\mathfrak{a}}(\omega_{n}) \right)_{n} \right) \\
    = &
    -3 \Bigl( \sigma_{\ell}(N\mathfrak{a}-\sigma_{N\mathfrak{a}}^{-1})\chi_{K}(\boldsymbol{\theta}_{d_{K}}^{\ast})\rho_{\et}(\boldsymbol{c}^{(d_{K})})+(1-\sigma_{\ell})
    (N\mathfrak{a}-\sigma_{N\mathfrak{a}}^{-1})\chi_{K}(\boldsymbol{\theta}_{d_{K}}^{\ast})\rho_{\et}(\boldsymbol{c}^{(d_{K})}) \\
    &\phantom{-3 \Bigl( \sigma_{\ell}(N\mathfrak{a}-\sigma_{N\mathfrak{a}}^{-1})\chi_{K}(\boldsymbol{\theta}_{d_{K}}^{\ast})\rho_{\et}(\boldsymbol{c}^{(d_{K})}\;)}+(1-\sigma_{-1})
    (N\mathfrak{a}-\sigma_{N\mathfrak{a}}^{-1})\boldsymbol{\theta}_{1}^{\ast}\rho_{\et}(\boldsymbol{c}^{(1)}) \Bigr) \\
    =& -3\left( (N\mathfrak{a}-\sigma_{N\mathfrak{a}}^{-1}) 
\chi_{K}(\boldsymbol{\theta}_{d_{K}}^{\ast})\rho_{\et}(\boldsymbol{c}^{(1)})
+(1-\sigma_{-1})(N\mathfrak{a}-\sigma_{N\mathfrak{a}}^{-1}) 
\boldsymbol{\theta}_{1}^{\ast}\rho_{\et}(\boldsymbol{c}^{(d_{K})})
\right)
\end{align*} modulo $\mathbb{Z}_{p} \cdot \rho_{\et}(\boldsymbol{\zeta})=H^{1}(\gal(K(\mu_{p^{\infty}})/K), \Lambda(1))$, as desired.
\end{proof}

\begin{rmk}\label{rmk:padicL}
We have two remarks regarding Theorem \ref{thm:main_measure}.
\begin{enumerate}
\item Since Theorem \ref{thm:Kersey} compares logarithms of absolute values of special values of the fundamental theta function and cyclotomic units, we cannot generalize Theorem \ref{thm:main_measure} to an exact equality in $H^{1}_{\et}(O_{K}[1/p], \Lambda(1))$. Such a generalization may be related to the notion of \emph{$\ell$-adic Galois L-functions} of Wojtkowiak \cite{Wo17}\footnote{We appreciate Densuke Shiraishi for informing us about the paper \cite{Wo17}.}. 
\item Using the Kubota-Leopoldt $p$-adic $L$-function $L_{p}$ \cite[Theorem 5.11]{Wa82}, we can rephrase the right-hand side of the equality in Theorem \ref{thm:main_measure} as follows: Let $K_{\infty}$ be the cyclotomic $\mathbb{Z}_{p}$-extension of $K$. We decompose the Galois group $\gal(K(\mu_{p^{\infty}})/K)$ as
\[
\gal(K(\mu_{p^{\infty}})/K)=\gal(K(\mu_{p})/K) \times \gal(K_{\infty}/K)= \Delta \times \Gamma,
\] and the $p$-adic cyclotomic character $\varepsilon_{\cyc}$ as $\varepsilon_{\cyc}(\sigma)=\omega(\sigma) \ab<\sigma>$, where $\omega$ is the Teichm\"{u}ller character and the character $\ab<  > \colon \Gamma \xrightarrow{\sim} 1+p\mathbb{Z}_{p}$ is defined by $ \ab<\sigma_{1+p}> \coloneqq 1+p$. For $i \in \mathbb{Z}/(p-1)\mathbb{Z}$ and $s \in \mathbb{Z}_{p}$, specializing the right-hand side of Theorem \ref{thm:main_measure} at the character $\omega^{i}\ab<  >^{s}$ gives
\begin{align*} 
\begin{dcases}
3(N\mathfrak{a}-\omega^{1-i}(\sigma_{N\mathfrak{a}})\ab<\sigma_{N\mathfrak{a}}>^{1-s})L_{p}(1-s, \chi_{K}\omega^{i})\rho_{\et,  \omega^{i}\ab<  >^{s}}(\boldsymbol{c}^{(1)}) & \text{if $i$ is odd, } \\
6(N\mathfrak{a}-\omega^{1-i}(\sigma_{N\mathfrak{a}})\ab<\sigma_{N\mathfrak{a}}>^{1-s})L_{p}(1-s, \omega^{i})\rho_{\et,  \omega^{i}\ab<  >^{s}}(\boldsymbol{c}^{(d_{K})}) & \text{if $i$ is even and $(i,s) \neq (0,0)$.} 
\end{dcases}
\end{align*} 
\end{enumerate}
\end{rmk}

Theorem \ref{thm:main} now follows immediately from Theorem \ref{thm:main_measure}:

\begin{cor}\label{cor:main} For each $m \geq 2$, we have the following equality between characters on $G_{K(p^{\infty})}$:
\[
\chi_{m, \mathfrak{a}}^{\E}=
\begin{dcases}
-6(N\mathfrak{a}-N\mathfrak{a}^{1-m})(1-p^{m-1})\frac{B_{m}}{m}\chi_{m}^{(d_{K})} & \text{if $m$ is even,} \\
-3(N\mathfrak{a}-N\mathfrak{a}^{1-m})(1-p^{m-1})\frac{B_{m, \chi_{K}}}{m}\chi_{m} & \text{if $m$ is odd.}
\end{dcases}
\]
\end{cor}

\begin{proof}
The assertion follows from Theorem \ref{thm:main_measure} and \cite[Theorem 12.2]{Wa82}.
\end{proof}

\begin{cor}[cf. Theorem \ref{thm:main}]\label{cor:main2}
For each $m \geq 2$, we have the following equality between characters on $G_{K(p^{\infty})}$:
\[
\frac{\chi_{m}^{\E}}{1-p^{2m-2}}
=
\begin{dcases}
-6w_{K} \cdot \frac{B_{m}}{m} \cdot \frac{\chi_{m}^{(d_{K})}}{1-p^{m-1}} & \text{if $m$ is even,} \\
-3w_{K} \cdot \frac{B_{m,\chi_{K}}}{m} \cdot \frac{\chi_{m}}{1-p^{m-1}} & \text{if $m$ is odd.}
\end{dcases}
\]
\end{cor}

\begin{proof}
The assertion follows from Proposition \ref{prp:comp} and Corollary \ref{cor:main}.
\end{proof}

\subsection{Proofs of corollaries}\label{3.2}

In this last section, we give two corollaries of Theorem \ref{thm:main}. First, we study reductions of elliptic Soul\'e characters modulo $p$. Note that certain sufficient conditions under which the reductions of the elliptic Soul\'e characters $\chi^{\E}_{(m_{1}, m_{2})}$ modulo $p$ are nontrivial are also studied in \cite[\textsection 5]{Is23}. Since we assume $p \geq 5$, we have
\[
\mathrm{Im}(\chi_{m}^{\E})
=
\begin{dcases}
\mathrm{Im}(\frac{B_{m}}{m} \cdot \chi_{m}^{(d_{K})}) \subset \mathbb{Z}_{p}(m) & \text{if $m \geq 2$ is even,} \\
\mathrm{Im}(\frac{B_{m,\chi_{K}}}{m} \cdot \chi_{m}) \subset \mathbb{Z}_{p}(m) & \text{if $m \geq 3$ is odd.}
\end{dcases}
\]
First, we consider two exceptional cases.

\begin{itemize}
\item If $m \equiv 1 \bmod p-1$, it is known \cite[Proposition 5.2]{Is23} that $\chi^{\E}_{m}$ is not surjective. On the other hand, $\chi_{m}$ is surjective since $\chi_{m} \bmod p$ corresponds to $p$-th roots of 
\[
\prod_{0<a<p} (1-\zeta_{p}^{a})=p \in K(p^{\infty})^{\times}/K(p^{\infty})^{\times p},
\] which is not trivial since $K(p^{\infty})$ is an abelian extension of $K$, whereas $K(\mu_{p}, \sqrt[p]{p})$ is not. The non-surjectivity of $\chi^{\E}_{m}$ comes from the fact that $B_{m, \chi_{K}}/m$ is divisible by $p$ \cite[Theorem 6]{Ca59}. 
\item If $m \equiv 0 \bmod p-1$, write $m=m'p^{n}(p-1)$ for some $n \geq 1$ and $(m',p)=1$. Then $p^{n+1}B_{m}/m$ is contained in $\mathbb{Z}_{p}^{\times}$ by the von Staudt-Clausen theorem \cite[Theorem 5.10]{Wa82}. This is in line with the fact that $\mathrm{Im}(\chi_{m}^{(d_{K})} \mid_{G_{K(p^{\infty})}}) \subset p^{n+1}\mathbb{Z}_{p}(m)$: By the five-term exact sequence, the cokernel of an injection
\[
H^{1}_{\et}(O_{K}[1/p], \mathbb{Z}_{p}(m))/H^{1}_{\et}(O_{K}[1/p], \mathbb{Z}_{p}(m))_{\mathrm{tors}} \hookrightarrow H^{1}_{\et}(O_{K(p^{\infty})}[1/p], \mathbb{Z}_{p}(m))^{\gal(K(p^{\infty})/K)}
\] induced by the restriction map is isomorphic to $H^{2}(\gal(K(p^{\infty})/K), \mathbb{Z}_{p}(m)) \cong \mathbb{Z}/p^{n+1}\mathbb{Z}$. Moreover, a similar argument to \cite[Lemma 5.5]{Is23} shows that 
\[
H^{1}_{\et}(O_{K(p^{\infty})}[1/p], \mathbb{Z}_{p}(m))^{\gal(K(p^{\infty})/K)} \cong \mathbb{Z}_{p}.
\] Hence the image of $\chi_{m}^{(d_{K})}$ (as a character on $G_{K(p^{\infty})}$) is contained in $p^{n+1}\mathbb{Z}_{p}$.

It follows from this equality that $\chi_{m}^{\E}$ is surjective if and only if $\chi_{m}^{(d_{K})}$ generates $H^{1}_{\et}(O_{K}[1/p], \mathbb{Z}_{p}(m))$ modulo the torsion subgroup. Since the restriction map induces an isomorphism
\begin{align*}
& H^{1}_{\et}(O_{K}[1/p], \mathbb{Z}_{p}(m))/H^{1}_{\et}(O_{K}[1/p], \mathbb{Z}_{p}(m))_{\mathrm{tors}} \\ \xrightarrow{\sim}& H^{1}_{\et}(O_{K(\mu_{p^{\infty}})}[1/p], \mathbb{Z}_{p}(m)) ^{\gal(K(\mu_{p^{\infty}})/K)} \cong \mathbb{Z}_{p}
\end{align*} (the second isomorphism follows from \cite[Theorem 2]{IS87}, for example), this is also equivalent to the surjectivity of $\chi_{m}^{(d_{K})}$ as a character on $G_{K(\mu_{p^{\infty}})}$.
\end{itemize}

For the remaining cases, we prove the following criterion:

\begin{cor}\label{cor:main3}
    The following two assertions are equivalent.
    \begin{enumerate}
        \item $\chi_{m}^{\E}$ is surjective for $m = 2,3, \dots, p-2$.
        \item The class number of $K(\mu_{p})$ is not divisible by $p$.
    \end{enumerate}
\end{cor}
\begin{proof}
First, assume that the assertion (1) holds. It follows by Corollary \ref{cor:main2} that $B_{m}$ is not divisible by $p$ for $m=2,4,\dots,p-2$ and that $B_{m,\chi_{K}}$ is not divisible by $p$ for $m=3,5, \dots, p-3$. Moreover, since we assume that $K$ has class number one, $B_{1, \chi_{K}}$ is also not divisible by $p$ by the class number formula for $K$. Hence the class number of $K(\mu_{p})$ is not divisible by $p$ by \cite[Theorem 1]{HP84}. 

Conversely, assume that (2) holds. Then $\chi_{m}$ is surjective for every odd $m \geq 3$ as a character on $G_{K(\mu_{p^{\infty}})}$ and so is $\chi_{m}^{(d_{K})}$ for every even $m \geq 2$, cf. Remark \ref{rmk:genSoule} (3). If $\chi_{m}$ is not surjective when restricted to $G_{K(p^{\infty})}$ for some odd $m \geq 3$, then the reduction of $\chi_{m}$ modulo $p$ would induce a nontrivial $\gal(K(p^{\infty})/K)$-equivariant character
\[
\gal(K(p^{\infty})/K(\mu_{p^{\infty}})) \to \mathbb{F}_{p}(m).
\] This is a contradiction since $\gal(K(p^{\infty})/K)$ acts trivially on the left-hand side, whereas it does not on the right-hand side. Hence $\chi_{m}$ is still surjective even restricted on $G_{K(p^{\infty})}$, and the same argument shows that $\chi_{m}^{(d_{K})}$ is so for $m=2,4, \dots, p-3$. Since (generalized) Bernoulli numbers appearing in Corollary \ref{cor:main2} are not divisible by $p$ for $m=2,3,\dots,p-2$ by \cite[Theorem 1]{HP84}, the assertion (1) holds.
\end{proof}

\begin{ex}
    We give examples of pairs $(K,p)$ of an imaginary quadratic field of class number one and a prime splitting in $K$ such that the class number of $K(\mu_{p})$ is not divisible by $p$, by using \cite[Theorem 1]{HP84}. By Corollary \ref{cor:main3}, the characters $\chi_{2}^{\E}, \chi_{3}^{\E}, \dots, \chi_{p-2}^{\E}$ are surjective for every such pairs.

    When $K=\mathbb{Q}(\sqrt{-1})$, the first few examples of split primes $p$ such that the class number of $K(\mu_{p})=\mathbb{Q}(\mu_{4p})$ does not divide $p$ are $5, 13, 17, 29, 41$ and $53$. The prime $37$ is excluded since it is irregular, and $61$ is the first regular prime splitting in $K$ such that the class number of $K(\mu_{p})$ is divisible by $p$. In fact, we have $61 \mid B_{7,\chi_{K}}=427/2$.
    
    Similarly, the prime $103$ is the first regular prime splitting in $K=\mathbb{Q}(\sqrt{-3})$ such that the class number of $K(\mu_{p})=\mathbb{Q}(\mu_{3p})$ is divisible by $p$. Finally, when $K=\mathbb{Q}(\sqrt{-2})$, the first few primes satisfying Corollary \ref{cor:main3} are $p=7, 17, 31, 41$ and $47$. The prime $23$ is excluded since $B_{11,\chi_{K}}$ is divisible by $23$.
\end{ex}

Secondly, we prove an analogue of the Coleman-Ihara formula \cite[p.105, (Col2)]{Ih86}. Fix a prime of $\overline{K}$ above $\p$. For each integer $n \geq 1$, let $U_{n}$ be the group of principal units $U(\mathbb{Q}_{p}(\mu_{d_{K}p^{n}}))$ of $\mathbb{Q}_{p}(\mu_{d_{K}p^{n}})$. For each integer $m \geq 2$, we consider the following composite
\[
U_{\infty} \coloneqq \varprojlim U_{n} \to \varprojlim U(\mathbb{Q}_{p}(\mu_{p^{n}})) \xrightarrow{\phi_{m}^{\mathrm{CW}}} \mathbb{Z}_{p}(m)
\] where the first arrow is induced by norm maps from $\mathbb{Q}_{p}(\mu_{d_{K}p^{n}})$ to $\mathbb{Q}_{p}(\mu_{p^{n}})$ and the second one is the $m$-th Coates-Wiles homomorphism for $\mathbb{Q}_{p}$ \cite[Definition 3.5]{NSW17}. Note that inverse limits are taken with respect to norm maps. We denote this homomorphism by the same letter $\phi^{\mathrm{CW}}_{m}$. We also have a homomorphism
\[
\mathrm{rec}: U_{\infty} \to G_{\mathbb{Q}(\mu_{d_{K}p^{\infty}})}^{\mathrm{ab}}
\] obtained by composing the local reciprocity map with the natural map from the decomposition group at the fixed prime above $\p$ to the global Galois group. Then an analogue of the Coleman-Ihara formula is stated as follows:

\begin{cor}\label{cor:Coleman-Ihara} We have
\[
\frac{\chi_{m}^{\E}(\mathrm{rec}(\boldsymbol{u}))}{1-p^{2m-2}}
=
\begin{dcases}
3w_{K} \cdot L^{\mathrm{Katz}}_{p}(m, \omega^{1-m})\frac{d_{K}^{m}}{g(\chi_{K})} \frac{\phi^{\mathrm{CW}}_{m}(\mathrm{rec}(\boldsymbol{u}))}{p^{m-1}-1} & \text{if $m \geq 2$ is even,} \\
3w_{K} \cdot L^{\mathrm{Katz}}_{p}(m, \omega^{1-m}) \frac{\phi^{\mathrm{CW}}_{m}(\mathrm{rec}(\boldsymbol{u}))}{p^{m-1}-1} & \text{if $m \geq 3$ is odd.} 
\end{dcases}
\] Here, $L^{\mathrm{Katz}}_{p}$ is Katz's $p$-adic L-function \cite{Ka78} restricted on the Galois group $\gal(K(\mu_{p^{\infty}})/K)$ and $g(\chi_{K}) \coloneqq \sum_{b \in \mathbb{Z}/d_{K}\mathbb{Z}} \chi_{K}(b)\zeta_{d_{K}}^{-b} $ is the Gauss sum associated to $\chi_{K}$.
\end{cor}
\begin{proof}
    First, let us assume that $m$ is odd. Then the Coleman-Ihara formula \cite[p.105, (Col2)]{Ih86} implies the equality
    \[
    \frac{\chi_{m}(\mathrm{rec}(\boldsymbol{u}))}{p^{m}-1}=L_{p}(m,\omega^{1-m})\phi_{m}^{\mathrm{CW}}(\boldsymbol{u}).
    \] Hence the assertion follows from Theorem \ref{thm:main} and the following factorization formula due to Gross \cite[Theorem]{Gr80}:
    \[
    L_{p}^{\mathrm{Katz}}(m, \omega^{1-m})=L_{p}(1-m, \chi_{K}\omega^{m})L_{p}(m, \omega^{1-m}).
    \] If $m$ is even, we apply the polylogarithmic analogue of the Coleman-Ihara formula, which is proved by Nakamura-Sakugawa-Wojtkowiak \cite[Theorem 4.4 (1)]{NSW17}, to the character $\chi_{m}^{(d_{K})}$. Note that, in terms of their restricted $p$-adic polylogarithmic character $\chi_{m}^{z}$ \cite[(2.10)]{NSW17}, we have
    \[
    \chi_{m}^{(d_{K})}=d_{K}^{m-1}\sum_{\substack{a \in \mathbb{Z}/d_{K}\mathbb{Z} \\ \chi_{K}(a)=1}}\chi_{m}^{\zeta_{d_{K}}^{a}}.
    \] by the definition of $\chi_{m}^{(d_{K})}$ and \cite[(2.11)]{NSW17}. Hence applying \cite[Theorem 4.4 (1)]{NSW17} to the character $\chi_{m}^{(d_{K})}$ gives the following equality:
    \[
    \chi_{m}^{(d_{K})}(\mathrm{rec}(\boldsymbol{u}))=d_{K}^{m-1}\mathrm{Tr}_{\mathbb{Q}_{p}(\zeta_{d_{K}})/\mathbb{Q}_{p}}\left(\left( \sum_{\substack{a \in \mathbb{Z}/d_{K}\mathbb{Z} \\ \chi_{K}(a)=1}}\mathrm{Li}^{(p)}_{m}(\zeta_{d_{K}}^{a}) \right) (1-p^{m-1}\sigma_{\mathbb{Q}_{p}(\zeta_{d_{K}})}) \phi_{m,\mathbb{Q}_{p}(\zeta_{d_{K}})}^{\mathrm{CW}}(\boldsymbol{u})\right).
    \] Here, $\mathrm{Tr}$ denotes the trace map, $\mathrm{Li}^{(p)}_{m}$ denotes the $p$-adic polylogarithm function (which is denoted by $\ell^{(p)}_{m}(z)$ in \cite{Col82}), $\sigma_{\mathbb{Q}_{p}(\zeta_{d_{K}})}$ denotes the Frobenius element and $\phi_{m,\mathbb{Q}_{p}(\zeta_{d_{K}})}^{\mathrm{CW}}: U_{\infty} \to \mathbb{Q}_{p}(\zeta_{d_{K}})\otimes \mathbb{Z}_{p}(m)$ is the $m$-th Coates-Wiles homomorphism \cite[Definition 3.5]{NSW17} having the property that $\mathrm{Tr}_{\mathbb{Q}_{p}(\mu_{d_{K}})/\mathbb{Q}_{p}}(\phi^{\mathrm{CW}}_{m, \mathbb{Q}_{p}(\zeta_{d_{K}})}(\boldsymbol{u}))=\phi_{m}^{\mathrm{CW}}(\boldsymbol{u})$ which follows from the characterization of the Coleman power series. Note that the right-hand side is simplified as
    $(1-p^{m-1}) \left( \sum_{\chi_{K}(a)=1}\mathrm{Li}^{(p)}_{m}(\zeta_{d_{K}}^{a}) \right)\phi^{\mathrm{CW}}_{m}(\boldsymbol{u})
    $ by this property. Now we prove the following claim:
    
    \medskip 
    
    {\bf Claim. } We have
    \[
     \sum_{\chi_{K}(a)=1}\mathrm{Li}^{(p)}_{m}(\zeta_{d_{K}}^{a})
     =
     -\frac{d_{K}}{2g(\chi_{K})}L_{p}(m, \chi_{K}\omega^{1-m}).
    \]
    \begin{proof}[Proof of the claim]
        By Coleman's formula \cite[p.172 (3), see also p.203]{Col82}, we have
        \[
        g(\chi_{K})d_{K}^{-1}\sum_{a \in (\mathbb{Z}/d_{K}\mathbb{Z})^{\times}}\chi_{K}(a)\mathrm{Li}^{(p)}_{m}(\zeta_{d_{K}}^{-a})=L_{p}(m,\chi_{K}\omega^{1-m}).
        \] Since $\mathrm{Li}^{(p)}_{m}(\zeta_{d_{K}}^{a})=-\mathrm{Li}^{(p)}_{m}(\zeta_{d_{K}}^{-a})$ by \cite[the proof of Proposition 6.4 (i)]{Col82}, we have
        \[
        \sum_{a \in (\mathbb{Z}/d_{K}\mathbb{Z})^{\times}}\chi_{K}(a)\mathrm{Li}^{(p)}_{m}(\zeta_{d_{K}}^{-a})= -2 \sum_{\chi_{K}(a)=1}\mathrm{Li}^{(p)}_{m}(\zeta_{d_{K}}^{a}).
        \] Hence the claim follows.
    \end{proof} 
    
    By this claim, we have the equality
    \[
    \frac{\chi_{m}^{(d_{K})}(\mathrm{rec}(\boldsymbol{u}))}{p^{m-1}-1}=\frac{d_{K}^{m}}{2g(\chi_{K})}L_{p}(m, \chi_{K}\omega^{1-m})\phi_{m}^{\mathrm{CW}}(\boldsymbol{u}).
    \]
    Using Theorem \ref{thm:main} and Gross's factorization formula, the desired equality follows.
    \end{proof}

\begin{rmk} In view of Nakamura-Sakugawa-Wojtkowiak \cite{NSW17}, it may be desirable to have an elliptic-polylogarithmic analogue of Corollary \ref{cor:Coleman-Ihara}. We hope to study such a generalization in a future article.
\end{rmk}

{\bf Acknowledgements.} The authors would like to thank Densuke Shiraishi and Kenji Sakugawa for their helpful comments and discussions. They are also grateful to the anonymous referee for valuable feedback. The first author was supported by JSPS KAKENHI Grant Number JP23KJ1882. The second author was supported by JST SPRING, Grant Number JPMJSP2123 and RIKEN Junior Research Associate Program.

\bibliographystyle{plain}
\bibliography{references}

\end{document}